\documentclass[12pt]{amsart}
\usepackage{geometry}     % See geometry.pdf to learn the layout options. There are lots.
\usepackage{graphicx}
\usepackage{color}
\usepackage{amsthm}
\usepackage{amsmath}
\usepackage{amsfonts}
\usepackage{enumerate}
\usepackage{amssymb}
\usepackage{epstopdf}
\usepackage{hyperref}
\usepackage{mathrsfs}

\newtheorem{theorem}{Theorem}[section]
\newtheorem{proposition}[theorem]{Proposition}
\newtheorem{lemma}[theorem]{Lemma}

\newtheorem{conjecture}[theorem]{Conjecture}

\DeclareMathOperator{\Vol}{Vol}

\DeclareMathOperator{\Hess}{Hess}
\DeclareMathOperator{\inj}{inj}
\DeclareMathOperator{\dVol}{dVol}
\DeclareMathOperator{\dL}{dL}
\DeclareMathOperator{\length}{L}
\DeclareMathOperator{\Image}{Im}
\DeclareMathOperator{\trace}{tr}
\DeclareMathOperator{\eval}{ev}
\DeclareMathOperator{\spectrum}{spec}
\DeclareMathOperator{\support}{spt}
\DeclareMathOperator{\reg}{Reg}

\theoremstyle{definition}
\newtheorem{definition}[theorem]{Definition}
\newtheorem{notation}[theorem]{Notation}
\newtheorem{remark}[theorem]{Remark}

\usepackage{booktabs,amsmath}
\def\Xint#1{\mathchoice
    {\XXint\displaystyle\textstyle{#1}}%
    {\XXint\textstyle\scriptstyle{#1}}%
    {\XXint\scriptstyle\scriptscriptstyle{#1}}%
    {\XXint\scriptscriptstyle\scriptscriptstyle{#1}}%
      \!\int}
\def\XXint#1#2#3{{\setbox0=\hbox{$#1{#2#3}{\int}$}
    \vcenter{\hbox{$#2#3$}}\kern-.5\wd0}}
\def\dashint{\Xint-}

\def\YYint#1#2#3{{\setbox0=\hbox{$#1{#2#3}{\int}$}
    \lower1ex\hbox{$#2#3$}\kern-.46\wd0}}
\def\YYYint#1#2#3{{\setbox0=\hbox{$#1{#2#3}{\int}$}
    \lower0.35ex\hbox{$#2#3$}\kern-.48\wd0}}

\def\ZZint#1#2#3{{\setbox0=\hbox{$#1{#2#3}{\int}$}
    \raise1.15ex\hbox{$#2#3$}\kern-.57\wd0}}
\def\ZZZint#1#2#3{{\setbox0=\hbox{$#1{#2#3}{\int}$}
    \raise0.85ex\hbox{$#2#3$}\kern-.53\wd0}}

\title{On the equidistribution of closed geodesics and geodesic nets}
\author{Xinze Li and Bruno Staffa}
\begin{document}
\maketitle
\begin{abstract}
We show that given a closed $n$-manifold $M$, for a Baire-generic set of Riemannian metrics $g$ on $M$ there exists a sequence of closed geodesics that are equidistributed in $M$ if $n=2$; and an equidistributed sequence of embedded stationary geodesic nets if $n=3$. One of the main tools that we use is the Weyl Law for the volume spectrum for $1$-cycles, proved in \cite{LMN} for $n=2$ and in \cite{GL22} for $n=3$. We show that our proof of the equidistribution of stationary geodesic nets can be generalized for any dimension $n\geq 2$ provided the Weyl Law for $1$-cycles in $n$-manifolds holds.
\end{abstract}
\tableofcontents

\section{Introduction}

Marques, Neves and Song proved in \cite{MNS} that for a generic set of Riemannian metrics in a closed manifold $M^{n}$, $3\leq n\leq 7$ there exists a sequence of closed, embedded, connected minimal hypersurfaces which is equidistributed in $M$.
In this paper, we study the equidistribution of closed geodesics and stationary geodesic nets (which are $1$-dimensional analogs of minimal hypersurfaces) on a Riemannian manifold $(M^{n},g)$, $n\geq 2$. We prove the following two results, for dimensions $2$ and $3$ of the ambient manifold respectively:

\begin{theorem}\label{thm1}
Let $M$ be a closed $2$-manifold. For a Baire-generic set of $C^{\infty}$ Riemannian metrics $g$ on $M$, there exists a set of closed geodesics that is equidistributed in $M$. Specifically, for every $g$ in the generic set, there exists a sequence $\{\gamma_i:S^{1} \rightarrow M\}$ of closed geodesics in $(M, g)$, such that for every $C^\infty$ function $f:M \rightarrow \mathbb{R}$ we have
\begin{equation*}
    \lim_{k \rightarrow \infty}\frac{\sum_{i = 1}^k\int_{\gamma_i}f \dL_g}{\sum_{i = 1}^k \length_{g}(\gamma_i)} = \frac{\int_M f \dVol_g}{\Vol(M, g)}.
\end{equation*}
\end{theorem}

\begin{theorem}\label{thm2}
Let $M$ be a closed $3$-manifold.  For a Baire-generic set of $C^{\infty}$ Riemannian metrics $g$ on $M$, there exists a set of connected embedded stationary geodesic nets that is equidistributed in $M$. Specifically, for every $g$ in the generic set , there exists a sequence $\{\gamma_i:\Gamma_i \rightarrow M\}$ of connected embedded stationary geodesic nets in $(M, g)$, such that for every $C^\infty$ function $f:M \rightarrow \mathbb{R}$ we have
\begin{equation*}
    \lim_{k \rightarrow \infty}\frac{\sum_{i = 1}^k\int_{\gamma_i}f \dL_g}{\sum_{i = 1}^k \length_{g}(\gamma_i)} = \frac{\int_M f \dVol_g}{\Vol(M, g)}.
\end{equation*}
\end{theorem}

\begin{remark}
We have an equivalent notion of equidistribution for a sequence of closed geodesics or geodesic nets: we say that $\{\gamma_{i}\}_{i\in\mathbb{N}}$ is equidistributed in $(M,g)$ if for every open subset $U\subseteq M$ it holds
    \begin{equation*}
    \lim_{k\to\infty}\frac{\sum_{i=1}^{k}\length_{g}(\gamma_{i}\cap U)}{\sum_{i=1}^{k}\length_{g}(\gamma_{i})}=\frac{\Vol_{g}U}{\Vol_{g}M}.
\end{equation*}
\end{remark}

Theorem \ref{thm2} is, as far as the authors know, the first result on equidistribution of $k$-stationary varifolds in Riemannian $n$-manifolds for $k<n-1$ (i.e. in codimension greater than $1$).
Regarding Theorem \ref{thm1}, similar equidistribution results for closed geodesics have been proved for compact hyperbolic manifolds in \cite{Bowen} in 1972 and for compact surfaces with constant negative curvature in \cite{Pollicott} in 1985. More recently, those results were extended to non-compact manifolds with negative curvature in \cite{Schapira} and to surfaces without conjugate points in \cite{Climenhaga}. The four previous works have in common that they approach the problem from the dynamical systems point of view. In the present paper, we approach it using Almgren-Pitts min-max theory (as it was done in \cite{MNS} for minimal hypersurfaces). Additionally, Theorem \ref{thm1} is the first equidistribution result for closed geodesics on closed surfaces that is proved for generic metrics, without any restriction regarding the curvature of the metric or the presence of conjugate points.

Our proof is inspired by the ideas in \cite{MNS}. There are two key results used in \cite{MNS} to prove equidistribution of minimal hypersurfaces for generic metrics: the Bumpy Metrics Theorem of Brian White \cite{White} and the Weyl Law for the Volume Spectrum proved by Liokumovich, Marques and Neves in \cite{LMN}: given a compact Riemannian manifold $(M^{n},g)$ with $n\geq 2$ (possibly with boundary), we have
    \begin{equation*}
        \lim_{p \rightarrow \infty}\omega_p^{n-1}(M,g)p^{-\frac{1}{n}} = \alpha(n)\Vol(M,g)^{\frac{n - 1}{n}}
    \end{equation*}
for some constant $\alpha(n) > 0$. Here, given $1\leq k\leq n-1$ we denote by $\omega_{p}^{k}(M,g)$ the $k$-dimensional $p$-width of $M$ with respect to the metric $g$ (for background on this, see \cite{GuthMinMax}, \cite{LS}, \cite{MN} \cite{IMN}). It was conjectured by Gromov (see \cite[section~8.4]{Gromov02}) that the Weyl law can be extended to other dimensions and codimensions. In this work, we are interested in the case of $1$-dimensional cycles. The following is the most general version of the Weyl law we could expect for $1$-cycles.

\begin{conjecture}\label{Weyl law 1}
Let $(M^{n},g)$ be a closed $n$-dimensional manifold, $n\geq 2$. Then there exists a constant $\alpha(n,1)>0$ such that
\begin{equation*}
     \lim_{p \rightarrow \infty}\omega^1_p(M^n, g)p^{-\frac{n - 1}{n}} = \alpha(n, 1)\Vol(M^n, g)^{\frac{1}{n}}.
\end{equation*}
\end{conjecture}

So far, Conjecture \ref{Weyl law 1} has been proved for $n=2$ as a particular case of \cite{LMN} and recently for $n=3$ by Guth and Liokumovich in their work \cite{GL22}. In this article, we use those two versions of the Weyl law to prove Theorem \ref{thm1}  and Theorem \ref{thm2}; and we also use the Structure Theorem for Stationary Geodesic Networks proved by Staffa in \cite{Staffa} and the Structure Theorem of White (\cite{White}) for the case of embedded closed geodesics. The work of Chodosh and Mantoulidis in \cite{Chodosh} is used to upgrade the equidistribution result for stationary geodesic networks to closed geodesics in dimension $2$. The only obstruction to extend our proof of the equidistribution of stationary geodesic nets to arbitrary dimensions of the ambient manifold $M$ is that Conjecture \ref{Weyl law 1} has not been proved yet if $n>3$. As all the rest of our argument works for any dimension $n>3$, what we do is to prove the following result and then show that it implies Theorem \ref{thm1} and Theorem \ref{thm2}.

\begin{theorem}\label{thm3}
Let $M^n$, $n\geq 2$ be a closed manifold. Assume that the Weyl law for $1$-cycles in $n$-manifolds holds. Then for a Baire-generic set of $C^{\infty}$ Riemannian metrics $g$ on $M$, there exists a set of connected embedded stationary geodesic nets that is equidistributed in $M$. Specifically, for every $g$ in the generic set , there exists a sequence $\{\gamma_i:\Gamma_i \rightarrow M\}$ of connected embedded stationary geodesic nets in $(M, g)$, such that for every $C^\infty$ function $f:M \rightarrow \mathbb{R}$ we have
\begin{equation*}
    \lim_{k \rightarrow \infty}\frac{\sum_{i = 1}^k\int_{\gamma_i}f \dL_g}{\sum_{i = 1}^k \length_{g}(\gamma_i)} = \frac{\int_M f \dVol_g}{\Vol(M, g)}.
\end{equation*}
\end{theorem}

In order to simplify the exposition, we consider integrals of $C^{\infty}$ functions instead of the  more general traces of $2$-tensors discussed in \cite{MNS}. Next we proceed to describe the intuition behind the proof, the technical issues which appear when one tries to carry on that intuition and how to sort them.

Let $g$ be a Riemannian metric on $M$. We want to do a very small perturbation of $g$ to obtain a new metric $\hat{g}$ which admits a sequence of equidistributed stationary geodesic networks. Let $f:M\to\mathbb{R}$ be a smooth function. Consider a conformal perturbation $\hat{g}:(-\delta,\delta)\to\mathcal{M}^{\infty}$ (for some $\delta>0$ small) defined as
\begin{equation*}
    \hat{g}(t)=e^{2tf}g.
\end{equation*}
By \cite[Lemma~3.4]{LS} the normalized $p$-widths $t\mapsto p^{-\frac{n-1}{n}}\omega_{p}^{1}(M,\hat{g}(t))$ are uniformly locally Lipschitz. This combined with the Weyl Law (recall that we assume it holds) implies that the sequence of functions $h_{p}:(-\delta,\delta)\to \mathbb{R}$
\begin{equation*}
    h_{p}(t)=\frac{p^{-\frac{n-1}{n}}\omega_{p}^{1}(M,\hat{g}(t))}{\Vol(M,\hat{g}(t))^{\frac{1}{n}}}
\end{equation*}
converges uniformly to the constant  $ \alpha(n,1)$. Considering
\begin{equation*}
    \tilde{h}_{p}(t)=\log(h_{p}(t))=-\frac{n-1}{n}\log(p)+\log(\omega_{p}(M,\hat{g}(t)))-\frac{1}{n}\log(\Vol(M,\hat{g}(t)))
\end{equation*}
we have that $\tilde{h}_{p}$ converges uniformly to the constant $\log(\alpha(n,1))$. On the other hand, Almgren showed that there is a correspondence between $1$-widths and the volumes of stationary varifolds (see \cite{FA62}, \cite{FA65}, 
\cite{CC92}, \cite{NR04}, \cite{JP73}, \cite{JP81}) such that for each $p\in\mathbb{N}$ and $t\in(-\delta,\delta)$ there exists a (possibly non unique) stationary geodesic network $\gamma_{p}(t)$ such that
\begin{equation}\label{Almgren widths}
    \length_{\hat{g}(t)}(\gamma_{p}(t))=\omega_{p}^{1}(\hat{g}(t)).
\end{equation}
Assume that the $\gamma_{p}(t)$'s can be chosen so that all of them are parametrized by the same graph $\Gamma$ and the maps $(-\delta,\delta)\to\Omega(\Gamma,M)$, $t\mapsto\gamma_{p}(t)$ are differentiable (this is a very strong assumption and doesn't necessarily hold, as the map $t\mapsto \omega_{p}^{1}(\hat{g}(t))$ may not be differentiable; a counterexample is shown below). In that case we can differentiate $\tilde{h}_{p}$ and obtain
\begin{align*}
    \frac{d}{dt}\tilde{h}_{p}(t) & = \frac{1}{\omega_{p}(M,\hat{g}(t))}\frac{d}{dt}\omega_{p}^{1}(M,\hat{g}(t))-\frac{1}{n\Vol(M,\hat{g}(t))}\frac{d}{dt}\Vol(M,\hat{g}(t))\\
    & = \frac{1}{L_{\hat{g}(t)}(\gamma_{p}(t))}\int_{\gamma_{p}(t)}f\dL_{\hat{g}(t)}-\frac{1}{n\Vol(M,\hat{g}(t))}\int_{M}nf\dVol_{\hat{g}(t)}\\
    & =\dashint_{\gamma_{p}(t)}f\dL_{\hat{g}(t)}-\dashint_{M}f\dVol_{\hat{g}(t)}.
\end{align*}

As $\{\tilde{h}_{p}\}_{p}$ converges uniformly to a constant, we could expect that the sequence $\{\tilde{h}_{p}'(t)\}_{p}$ converges to $0$ for some values of $t$. If that was the case, the sequence $\{\gamma_{p}(t)\}_{p}$ would verify the equidistribution formula for the function $f$ with respect to the metric $\hat{g}(t)$. Nevertheless, this does not have to be true, because of two reasons. The first one is that the uniform convergence of a sequence of functions to a constant does not imply convergence of the derivatives to $0$ at any point. Indeed, we can construct a sequence of zigzag functions which converges uniformly to $0$ but $h_{p}'(t)$ does not converge to $0$ for any $t$. The second one is that the differentiability of $t\mapsto\gamma_{p}(t)$ could fail, a counterexample is shown in the next paragraph. And even if that reasoning was true and such $t$ existed, the sequence $\{\gamma_{p}(t)\}_{p}$ constructed would only give an equidistribution formula for the function $f$ (which is used to construct the sequence) instead of for all $C^{\infty}$ functions at the same time; and with respect to a metric $\hat{g}(t)$ which could also vary with $f$.

An example when $t\mapsto\omega_{p}^{1}(\hat{g}(t))$ is not differentiable is the following. Let us consider a dumbbell metric $g$ on $S^{2}$ obtained by constructing a connected sum of two identical round $2$-spheres $S^{2}_{1}$ and $S^{2}_{2}$ of radius $1$ by a thin neck. Define a $1$-parameter family of metrics $\{\hat{g}(t)\}_{t\in(-1,1)}$ such that $\hat{g}(t)=(1+t)^{2}g$ along $S^{2}_{1}$, $\hat{g}(t)=(1-t)^{2}g$ along $S^{2}_{2}$ (interpolating along the neck so that it is still very thin). It is clear than for $t\geq 0$, the $1$-width is realized by a great circle in $S^{2}_{1}$ with length $1+t$, and for $t\leq 0$ it is realized by a great circle in $S^{2}_{2}$ of length $1-t$. Therefore

\[\omega_{1}^{1}(\hat{g}(t))=
\begin{cases}
1-t & \text{if }t\leq 0\\
1+t & \text{if }t\geq 0
\end{cases}
\]
and hence it is not differentiable at $0$.

To fix the previous issue (differentiability of $\omega^{1}_{p}(g(t))$), we prove Proposition \ref{Perturbation} which is a version for stationary geodesic networks of \cite[Lemma~2]{MNS}. Regarding the convergence of $h_{p}'(t)$ to $0$ for certain values of $t$, we use Lemma \ref{gradient} which is exactly \cite[Lemma~3]{MNS}. To obtain a sequence of stationary geodesic networks that verifies the equidistribution formula for all $C^{\infty}$ functions (and not only for a particular one as above), we carry on a construction described in Section \ref{proof main thm} using certain stationary geodesic nets which realize the $p$-widths in a similar way as the $\gamma_{p}(t)$'s above. The key idea here is that the integral of any $C^{\infty}$ function $f$ over $M$ can be approximated by Riemann sums along small regions with piecewise smooth boundary where $f$ is almost constant. Therefore, if we have an equidistribution formula for the characteristic functions of those regions (or some suitable smooth approximations), then we will be able to deduce it for an arbitrary $f\in C^{\infty}(M,\mathbb{R})$. The advantage of doing this is that we reduce the problem to a countable family of functions. This argument is also inspired by \cite{MNS}.

The paper is structured as follows. In Section \ref{geodesic networks}, we introduce the set up and necessary preliminaries. In Section \ref{Section Jacobi}, we define the Jacobi Operator along a stationary geodesic net and show that it has all the nice properties that an elliptic operator has (mainly, admitting an orthonormal basis of eigenfunctions and therefore having a min-max characterization for its eigenvalues). This is crucial to prove Proposition \ref{Perturbation nondeg}. In Section \ref{two technical propositions}, the technical propositions necessary to prove Theorem \ref{thm3} are discussed. In Section \ref{proof main thm} we prove Theorem \ref{thm3} and get Theorem \ref{thm2} as a corollary using the Weyl law for $1$ cycles in $3$ manifolds proved in \cite{GL22}. In Section \ref{closed geodesics}, we use the Weyl law from \cite{LMN} and the proof of Theorem \ref{thm3} combined with the work of Chodosh and Mantoulidis in \cite{Chodosh} (where it is proved that the $p$-widths on a surface are realized by finite unions of closed geodesics) to prove Theorem \ref{thm1}.

\begin{remark}\label{rkrohil}
Rohil Prasad pointed out that an alternative proof of Theorem \ref{thm1} could be obtained using the methods of Irie in \cite{Irie}. Given a closed Riemannian $2$-manifold $(M,g)$, its unit cotangent bundle $U^{*}_{g}M$ is a closed $3$-manifold equipped with a natural contact structure induced by the contact form $\lambda_{g}$ which is the restriction of the Liouville form $\lambda$ on $T^{*}M$ to $U^{*}_{g}M$. It is a well known fact that the Reeb vector field associated to $\lambda_{g}$ generates the geodesic flow of $(M,g)$. Additionally, given a function $f:M\to\mathbb{R}$, the Riemannian metric $g'=e^{f}g$ corresponds to the conformal perturbation $\lambda_{g'}=e^{\frac{f\circ\pi}{2}}\lambda_{g}$ of the contact form in $U^{*}_{g}M$ (here $\pi:U^{*}_{g}M\to M$ is the projection map); and both $\lambda_{g}$ and $\lambda_{g'}$ are compatible with the same contact structure on $U^{*}_{g}M$. Thus one would like to apply \cite[Corollary~1.4]{Irie} to $U^{*}_{g}M$ with the contact structure induced by $\lambda_{g}$. However, that result is about generic perturbations of the contact form of the type $e^{\tilde{f}}\lambda_{g}$, where $\tilde{f}:U^{*}_{g}M\to\mathbb{R}$ and we only want to consider perturbations $\tilde{f}=f\circ\pi$ which are liftings to $U^{*}_{g}M$ of maps $f:M\to\mathbb{R}$ so some work should be done here in order to apply Irie's result in our setting. This issue was pointed out in \cite[Remark~2.3]{Chen}, where a similar problem is studied for Finsler metrics and a solution is given for that class of metrics. Additionally, Irie's theorem would give us an equidistributed sequence for a generic conformal perturbation of each metric $g$. This immediately implies that for a dense set of Riemannian metrics such an equidistribution result holds, but some additional arguments are needed to prove it for a Baire-generic metric. It is important to point out that the result in \cite{Irie} uses the ideas of \cite{MNS} but in the different setting of contact geometry, applying results of Embedded Contact Homology with the purpose of finding closed orbits of the Reeb vector field; while in \cite{MNS} Almgren-Pitts theory is used to find closed minimal surfaces.
\end{remark}

\vspace{0.2in}

\textbf{Aknowledgements.} We are very grateful to Yevgeny Liokumovich for suggesting this problem and for his valuable guidance and support while we were working on it. We are deeply thankful to the referee for reading the article very carefully and for their great feedback, particularly for pointing out a gap in the proof of Proposition \ref{PropP} which was filled by proving Proposition \ref{Perturbation nondeg}. We also want to thank Rohil Prasad for his suggestion about using the approach mentioned in Remark \ref{rkrohil} to get an alternative proof of Theorem \ref{thm1} and for the valuable discussion derived from there. We are thankful to Wenkui Du as well, because of his comments on the preliminary version of this work. Bruno Staffa was partially supported by Discovery Grant RGPIN-2019-06912.

\section{Preliminaries}\label{geodesic networks}
\begin{definition}[Weighted Multigraph]
A weighted multigraph is a graph $\Gamma = (\mathscr{E}, \mathscr{V}, \{n(E)\}_{E \in \mathscr{E}})$ consisting of a set of edges $\mathscr{E}$, a set of vertices $\mathscr{V}$ and a multiplicity $n(E) \in \mathbb{N}$ assigned to each edge $E \in \mathscr{E}$. A weighted multigraph is \textbf{good} if it is connected and either it is a closed loop with mutiplicity or each vertex $v \in \mathscr{V}$ has at least three different incoming edges (here loop edges $E$ at $v$ count twice as an incoming edge at $v$, see \cite{Staffa} for a more detailed discussion). In the later case we say $\Gamma$ is \textbf{good*}.
\end{definition}

\begin{definition}
Given a weighted multigraph $(\mathscr{E},\mathscr{V},\{n(E)\}_{E\in\mathscr{E}})$, we identify each edge $E\in\mathscr{E}$ with the interval $[0,1]$ and we denote $\pi=\pi_{E}:\{0,1\}\to\mathscr{V}$ the map sending $i\in\{0,1\}$ to the vertex $v\in \mathscr{V}$ under the identification $E\cong[0,1]$.
\end{definition}

\begin{definition}[$\Gamma$-net]
A $\Gamma$-net $\gamma$ on $M$ is a continuous map $\gamma: \Gamma \rightarrow M$ which is a $C^2$-immersion restricted to the edges of $\Gamma$. We will denote $\Omega(\Gamma, M)$ the space of $\Gamma$-nets on $M$. It has a natural Banach manifold structure as a subspace of $\prod_{E\in\mathscr{E}} C^{2}(E,M)$ (see \cite{Staffa}). 
\end{definition}

\begin{definition}
We say that two $\Gamma$-nets $\gamma_{1}$ and $\gamma_{2}$ are equivalent if for every edge $E$ of $\Gamma$ the map $\gamma_{1}|_{E}$ is a $C^{2}$ reparametrization of $\gamma_{2}|_{E}$ fixing the endpoints. This defines an equivalence relation $\sim$ in $\Omega(\Gamma,M)$. We denote $\hat{\Omega}(\Gamma,M)=\Omega(\Gamma,M)/\sim$ the quotient space. Given $\gamma\in\hat{\Omega}(\Gamma,M)$ we will often denote a representative of the equivalence class $\gamma$ also by $\gamma$, and regard different representatives as different parametrizations of the geometric object $\gamma\in\hat{\Omega}(\Gamma,M)$.
\end{definition}

\begin{notation}
Given a $\Gamma$-net $\gamma$ and an edge $E\in\mathscr{E}$, we denote $\gamma_{E}$ the restriction of $\gamma$ to $E$. We also define $\gamma_{E}(0):=\gamma_{E}(\pi_{E}(0))$ and $\gamma_{E}(1):=\gamma_{E}(\pi_{E}(1))$.
\end{notation}

\begin{notation}
Given $1\leq q\leq\infty$, let us denote $\mathcal{M}^{q}$ the set of $C^{q}$ Riemannian metrics on $M$.
\end{notation}

\begin{definition}
Let $\gamma\in\hat{\Omega}(\Gamma,M)$ and let $h$ be a continuous function defined in $\Image(\gamma)\subseteq M$. Given a metric $g\in\mathcal{M}^{q}$ we define
\begin{equation*}
    \int_{\gamma}h\dL_{g} = \sum_{E \in \mathscr{E}}n(E)\int_{E}h\circ \gamma(u)\sqrt{g_{\gamma(u)}(\dot{\gamma}(u),\dot{\gamma}(u))}du.
\end{equation*}
Observe that the right hand side is independent of the parametrization we choose and therefore $\int_{\gamma}h\dL_{g}$ is well defined.
\end{definition}

\begin{definition}[$g$-Length]
Given $g \in \mathcal{M}^{q}$ and $\gamma \in \hat{\Omega}(\Gamma, M)$, we define the $g$-length of $\gamma$ by 
\begin{equation*}
    \length_g(\gamma) = \int_\Gamma 1\dL_{g}=\sum_{E\in\mathscr{E}}n(E)\int_{E}\sqrt{g_{\gamma(u)}(\dot{\gamma}(u),\dot{\gamma}(u))}du.
\end{equation*}
\end{definition}

\begin{definition}[Stationary Geodesic Network]
We say that $\gamma \in \Omega(\Gamma, M)$ is a stationary geodesic network with respect to a metric $g \in \mathcal{M}^{q}$ ($q\geq 2$) if it is a critical point of the length functional $\length_g:\Omega(\Gamma, M) \rightarrow \mathbb{R}$. This means that given any smooth one parameter family $\tilde{\gamma}:(-\delta,\delta)\to\Omega(\Gamma,M)$ with $\tilde{\gamma}(0)=\gamma$ we have
\begin{equation*}
    \frac{d}{ds}\bigg|_{s=0}\length_{g}(\tilde{\gamma}(s))=0.
\end{equation*}
Assuming that the edges of $\gamma$ are parametrized by constant speed, if $X(t)=\frac{\partial\tilde{\gamma}}{\partial s}(0,t)$ (here we regard $\tilde{\gamma}:(-\delta,\delta)\times\Gamma\to M$) then
\begin{equation}\label{First variation formula}
\frac{d}{ds}\bigg|_{s=0}\length_{g}(\tilde{\gamma}(s))=-\sum_{E\in\mathscr{E}}\frac{n(E)}{l(E)}\int_{E}\langle\ddot{\gamma}(t),X(t)\rangle_{g}dt+\sum_{v\in\mathscr{V}}\langle V_{v}(\gamma),X(v)\rangle_{g}
\end{equation}
where $l(E)=\length_{g}(\gamma_{E})$ and
\begin{equation*}
    V_{v}(\gamma)=\sum_{(E,i):\pi_{E}(i)=v}(-1)^{i+1}n(E)\frac{\dot{\gamma}_{E}(i)}{|\dot{\gamma}_{E}(i)|}.
\end{equation*}
Equation (\ref{First variation formula}) is called the First Variation Formula and was computed in \cite[Section~1]{Staffa}. It implies that $\gamma:\Gamma\to M$ is stationary with respect to $g$ if and only if each edge is mapped to a geodesic segment in $(M,g)$ and the stability condition at the vertices $V_{v}(\gamma)=0$ is verified. The latter means that for each $v\in\mathscr{V}$, the sum of the inward pointing unit tangent vectors to each edge at $v$ is $0$.
\end{definition}

\begin{definition}
We say that $\gamma\in\hat{\Omega}(\Gamma,M)$ is a stationary geodesic network if every representative $\tilde{\gamma}\in\Omega(\Gamma,M)$ of $\gamma$ is a stationary geodesic network.
\end{definition}

\begin{definition}
We denote $C^{2}(\gamma)$ the space of continuous vector fields along $\gamma$ whose restriction to each edge is of class $C^{2}$.
\end{definition}

\begin{remark}
If $g\in\mathcal{M}^{q}$, $q\geq 2$ and $\gamma\in\Omega(\Gamma,M)$ is stationary with respect to $g$ then by the regularity of the solutions of an ODE, $\gamma_{E}$ is of class $C^{q}$ for every $E\in\mathscr{E}$. This is why we only ask $C^{2}$ regularity to $\Gamma$-nets and vector fields along them.
\end{remark}

Assume $\gamma\in\Omega(\Gamma,M)$ is a stationary geodesic net with respect to a $C^{q}$ metric with $q\geq 3$ (so that the Riemann curvature tensor is of class $C^{1}$). Let $\tilde{\gamma}:(-\delta,\delta)^{2}\to\Omega(\Gamma,M)$ be a smooth $2$-parameter family of $\Gamma$-nets with $\tilde{\gamma}(0,0)=\gamma$. Let $X(t)=\frac{\partial \tilde{\gamma}}{\partial x}(0,0,t)$ and $Y(t)=\frac{\partial\tilde{\gamma}}{\partial s}(0,0,t)$. We define the Hessian $\Hess_{\gamma}\length_{g}:C^{2}(\gamma)\times C^{2}(\gamma)\to\mathbb{R}$ of the length functional at $\gamma$ as the bilinear form
\begin{equation*}
    \Hess_{\gamma}\length_{g}(X,Y)=\frac{\partial^{2}}{\partial x\partial s}\bigg|_{(0,0)}\length_{g}(\tilde{\gamma}(x,s)).
\end{equation*}
In \cite[Section~2]{Staffa} it was shown that $\Hess_{\gamma}\length_{g}$ is well defined (i.e. it does not depend on which two parameter variation $\tilde{\gamma}$ with directional derivatives $X$ and $Y$ we choose) and in fact it holds
\begin{equation}\label{Second Variation Formula}
    \Hess_{\gamma}\length_{g}(X,Y)=\sum_{E\in\mathscr{E}}\int_{E}\langle A_{E}(X)(t),Y(t)\rangle_{g}dt+\sum_{v\in\mathscr{V}}\langle B_{v}(X),Y(v)\rangle_{g}
\end{equation}
where
\begin{align*}
    A_{E}(X) & =-\frac{n(E)}{l(E)}[\ddot{X}_{E}^{\perp}+R(\dot{\gamma},X_{E}^{\perp})\dot{\gamma}]\\
    B_{v}(X) & =\sum_{(E,i):\pi_{E}(i)=v}(-1)^{i+1}\frac{n(E)}{l(E)}\dot{X}_{E}^{\perp}(i)
\end{align*}
being $X_{E}$ the restriction of the vector field $X$ to the edge $E$ and $X_{E}^{\perp}$ the component of $X_{E}$ orthogonal to $\gamma_{E}$. Observe that $A_{E}$ is (up to a positive constant) the Jacobi operator along $\gamma_{E}$. Equation (\ref{Second Variation Formula}) is the Second Variation Formula.

\begin{definition}
We say that a vector field $J\in C^{2}(\gamma)$ is Jacobi if it is a null vector of $\Hess_{\gamma}\length_{g}$, i.e. if $\Hess_{\gamma}\length_{g}(J,X)=0$ for every $X\in C^{2}(\gamma)$. By the Second Variation Formula, $J$ is Jacobi along $\gamma$ if and only if $A_{E}(J)=0$ for every $E\in\mathscr{E}$ and $B_{v}(J)=0$ for every $v\in\mathscr{V}$.
\end{definition}

\begin{definition}
A vector field $X\in C^{2}(\gamma)$ is said to be parallel if $X_{E}$ is parallel along $\gamma_{E}$ for every $E\in\mathscr{E}$.
\end{definition}

\begin{remark}
Observe that every parallel vector field $J$ along $\gamma$ is Jacobi.
\end{remark}

\begin{definition}
A stationary geodesic net $\gamma:\Gamma\to (M,g)$ is nondegenerate if every Jacobi field along $\gamma$ is parallel.
\end{definition}

\begin{remark}\label{unionofregular}
In \cite[Lemma~2.5]{LS}, it was shown that every stationary geodesic network with respect to a metric $g$ can be represented by a map $\gamma:\Gamma\to M$, where $\Gamma=\bigcup_{i=1}^{P}\Gamma_{i}$ is the finite union of the good weighted multigraphs $\{\Gamma_{i}\}_{1\leq i\leq P}$ and $\gamma|_{\Gamma_{i}}$ is an embedded stationary geodesic network for each $1\leq i\leq P$ (moreover, the map $\gamma:\Gamma\to M$ is a topological embedding).
\end{remark}

\begin{definition}\label{nondegcomponents}
Given a stationary geodesic network $\gamma:\Gamma\to (M,g)$, we say that its connected components are nondegenerate if
\begin{enumerate}
    \item We can express $\Gamma=\bigcup_{i=1}^{P}\Gamma_{i}$ as a disjoint union of good weighted multigraphs.
    \item $\gamma|_{\Gamma_{i}}$ is an embedded nondegenerate stationary geodesic network for every $1\leq i\leq P$.
\end{enumerate}
\end{definition}

\begin{definition}
An almost embedded closed geodesic in a Riemannian manifold $(M,g)$ is a map $\gamma:S^{1}\to (M,g)$ such that
\begin{enumerate}
    \item $\gamma$ is geodesic (i.e. $\ddot{\gamma}(t)=0$ for every $t\in S^{1}$).
    \item $\gamma$ is an immersion (i.e. $\dot{\gamma}(t)\neq 0$ for every $t\in S^{1}$).
    \item All self-intersections of $\gamma$ are transverse, which means that for every $s,t\in S^{1}$ such that $\gamma(s)=\gamma(t)$ and $s\neq t$, the velocities $\dot{\gamma}(s)$ and $\dot{\gamma}(t)$ are not colinear.
\end{enumerate}
This terminology is inspired in \cite[Definition~2.2]{WhiteCinfty}, where Brian White extends his Bumpy Metrics Theorem to almost embedded minimal surfaces.
\end{definition}

\begin{notation}
Given a symmetric $2$-tensor $T$, a metric $g\in\mathcal{M}^{q}$, a stationary geodesic network $\gamma:\Gamma\to M$ on $(M,g)$ and $t\in \Gamma$, we denote
\begin{equation*}
    \trace_{\gamma,g}T(t)=T(\frac{\dot{\gamma}(t)}{|\dot{\gamma}(t)|_{g}},\frac{\dot{\gamma}(t)}{|\dot{\gamma}(t)|_{g}})
\end{equation*}
which is the trace of the tensor $T$ along $\gamma$ with respect to the metric $g$.
\end{notation}

\begin{definition}[Average integral along $\gamma$]
Let $\Gamma$ be a weighted multigraph. Given $\gamma \in \Omega(\Gamma, M)$, a metric $g \in \mathcal{M}^{q}$ and a continuous function $h$ defined in $\Image(\gamma)$, we define the average integral of $h$ with respect to metric $g$ as
\begin{align*}
    \dashint_{\gamma} h \dL_{g} 
    &:= \frac{1}{\length_g(\gamma)}\int_{\gamma}h \dL_{g}.
\end{align*}
\end{definition}

\section{The Jacobi Operator}\label{Section Jacobi}

In this section we will study some properties of the Jacobi operator of an embedded stationary geodesic network $\gamma:\Gamma\to (M,g)$, where $\Gamma$ is a good weighted multigraph and $g\in\mathcal{M}^{q}$, $q\geq 3$. We will focus on the case when $\Gamma$ is good* (i.e. every vertex has at least three different incoming edges), because when $\Gamma$ is a loop with multiplicity what we get is the Jacobi operator along an embedded closed geodesic acting on normal vector fields, which is known to be elliptic; and hence it has all the nice properties that we will describe below. We first introduce some notation. Let
\begin{align*}
    C^{2}(\gamma) & =\{X\text{ continuous vector field along }\gamma:X_{E}\text{ is }C^{2}\text{ }\forall E\in\mathscr{E}\}\\
    C^{2}(\gamma)^{\parallel} & =\{X\in C^{2}(\gamma):X(t)\in\langle\dot{\gamma}(t)\rangle\text{ }\forall t\in\Gamma\}\\
    C^{2}_{0}(\gamma)^{\perp} & = \{X\in C^{2}(\gamma):X(t)\perp\dot{\gamma}(t)\text{ }\forall t\in\Gamma\setminus\mathscr{V}\text{ and }X(v)=0\text{ }\forall v\in\mathscr{V}\}\\
    C^{2}(E)^{\perp} &=\{X\in C^{2}(E):X(t)\perp\dot{\gamma}(t)\text{ }\forall E\in\mathscr{E}\}\\
    C^{2}(\mathscr{E})^{\perp} & =\prod_{E\in\mathscr{E}}C^{2}(E)^{\perp}.
\end{align*}
Observe that as $\Gamma$ is good*, if $X\in C^{2}(\gamma)^{\parallel}$ then $X(v)=0$ for every $v\in\mathscr{V}$. Denote
\begin{equation*}
    T\mathscr{V}=\prod_{v\in\mathscr{V}}T_{\gamma(v)}M.
\end{equation*}
By the second variation formula (\ref{Second Variation Formula}), we can define the Jacobi operator $L:C^{2}(\gamma)\to C^{0}(\mathscr{E})^{\perp}\times T\mathscr{V}$ as
\begin{equation}\label{Jacobi operator}
    L(J)=((-\frac{n(E)}{l(E)}(\ddot{J}_{E}^{\perp}+R(\dot{\gamma},J_{E}^{\perp})\dot{\gamma}))_{E\in\mathscr{E}},(B_{v}(J))_{v\in\mathscr{V}}).
\end{equation}
We know that each $X\in C^{2}(\gamma)^{\parallel}$ is Jacobi (i.e. it verifies $L(J)=0$). We want to construct a complement of $C^{2}(\gamma)^{\parallel}$ in $C^{2}(\gamma)$, and show that when we restrict $L$ to that complement it behaves like an elliptic operator (this complement will play the role of the space of normal Jacobi fields along a minimal submanifold in the smooth case, when it is known that the stability operator is elliptic).

%To do this, we will need to define a finite dimensional subspace $S^{2}(\gamma)\subseteq C^{2}(\gamma)$ such that the evaluation map $\eval:S^{2}(\gamma)\to T\mathscr{V}$, $J\mapsto (J(v))_{v\in\mathscr{V}}$ is a linear isomorphism. This can be done by taking a basis $\mathcal{B}_{v}$ of $T_{\gamma(v)}M$ for each $v\in\mathscr{V}$, and given an element $w=(w_{v})_{v\in\mathscr{V}}\in\prod_{v\in\mathscr{V}}\mathcal{B}_{v}$ define a vector field $J_{w}\in C^{2}(\gamma)$ such that $J_{w}(v)=w_{v}$ for each $v\in\mathscr{V}$. Then we can define $S^{2}(\gamma)=\langle J_{w}:w\in\prod_{v\in\mathscr{V}}\mathcal{B}_{v}\rangle$. Of course the choice of $S^{2}(\gamma)$ is not canonical, but we fix one choice and work with it for the rest of the section (it will be deduced from the arguments below that the results that we prove hold independently of the choice of $S^{2}(\gamma)$). It is clear that

To do this, we will need to define a finite dimensional subspace $S^{2}(\gamma)\subseteq C^{2}(\gamma)$ such that the evaluation map $\eval:S^{2}(\gamma)\to T\mathscr{V}$, $J\mapsto (J(v))_{v\in\mathscr{V}}$ is a linear isomorphism. This can be done by taking a basis $\mathcal{B}_{v}$ of $T_{\gamma(v)}M$ for each $v\in\mathscr{V}$, and for each pair $(v,w)$ with $v\in\mathscr{V}$ and $w\in\mathcal{B}_{v}$ defining a vector field $J_{(v,w)}\in C^{2}(\gamma)$ such that $J_{(v,w)}(v)=w$ and $J_{(v,w)}(v')=0$ for every $v'\neq v$. Then we can define $S^{2}(\gamma)=\langle J_{(v,w)}:v\in\mathscr{V},w\in\mathcal{B}_{v}\rangle$. Of course the choice of $S^{2}(\gamma)$ is not canonical, but we fix one choice and work with it for the rest of the section (it will be deduced from the arguments below that the results that we prove hold independently of the choice of $S^{2}(\gamma)$). It is clear that
\begin{equation*}
    C^{2}(\gamma) =C^{2}(\gamma)^{\parallel}\oplus C^{2}_{0}(\gamma)^{\perp}\oplus S^{2}(\gamma).
\end{equation*}
Denote $C^{2}(\gamma)^{C}=C^{2}_{0}(\gamma)^{\perp}\oplus S^{2}(\gamma)$ which is a complement of the space of parallel vector fields along $\gamma$. Same as in the theory of elliptic operators, we can extend the Jacobi operator to Sobolev spaces of vector fields along $\gamma$ once we have a suitable definition of them. Denote
\begin{align*}
    H^{2}_{0}(E) & =\{X\text{ normal vector field of class }H^{2}_{0}\text{ along }E\}\\
    H^{2}_{0}(\gamma) & =\prod_{E\in\mathscr{E}}H^{2}_{0}(E)\\
    H^{2}(\gamma) & = H^{2}_{0}(\gamma)\oplus S^{2}(\gamma)\\
    L^{2}(E) & =\{X\text{ normal vector field of class }L^{2}\text{ along }E\}\\
    L^{2}(\mathscr{E}) & =\prod_{E\in\mathscr{E}} L^{2}(E)\\
    L^{2}(\gamma) & = L^{2}(\mathscr{E})\oplus T\mathscr{V}.
\end{align*}
Notice that $H^{2}(\gamma)$ is the $H^{2}$-version of $C^{2}(\gamma)^{C}$ and will be the domain of the Jacobi operator we will work with (as that operator vanishes on $C^{2}(\gamma)^{\parallel}$). The previous spaces are defined in analogy with the spaces of $C^{2}$, $H^{2}$ and $L^{2}$ normal vector fields along a smooth closed submanifold which appear when studying the ellipticity of its Jacobi operator. The space $L^{2}(\gamma)$ is a Hilbert space with the inner product
\begin{equation*}
    \langle ((X_{E})_{E},(u_{v})_{v}),((Y_{E})_{E},(w_{v})_{v})\rangle=\sum_{E\in\mathscr{E}}\int_{E}\langle X_{E}(t),Y_{E}(t)\rangle_{g}dt+\sum_{v\in\mathscr{V}}\langle u_{v},w_{v}\rangle_{g}
\end{equation*}
and we have a monomorphism $\iota: H^{2}(\gamma)\to L^{2}(\gamma)$ with dense image given by
\begin{equation*}
    \iota(J)=((J_{E})^{\perp}_{E\in\mathscr{E}}, (J(v))_{v\in\mathscr{V}})
\end{equation*}
which allow us to write the Hessian $\Hess_{\gamma}\length_{g}:H^{2}(\gamma)\times H^{2}(\gamma)\to\mathbb{R}$ as
\begin{equation*}
    \Hess_{\gamma}\length_{g}(J,\tilde{J})=\langle L(J),\iota(\tilde{J})\rangle
\end{equation*}
where $\langle ,\rangle$ is the inner product in $L^{2}(\gamma)$. Here we considered $L:H^{2}(\gamma)\to L^{2}(\gamma)$ given by (\ref{Jacobi operator}) which is a bounded linear operator.
As in the smooth case, we can also regard $L$ as an unbounded operator $L:L^{2}(\gamma)\to L^{2}(\gamma)$ whose domain is the dense linear subspace $H^{2}(\gamma)$. 
We would therefore expect that for sufficiently big $\lambda\in\mathbb{R}$ the operator $L+\lambda\iota:L^{2}(\gamma)\to L^{2}(\gamma)$ has a compact inverse, and from that get an orthonormal basis of $L^{2}(\gamma)$ consisting of eigenvectors of $L$. This indeed holds, as it is shown in the following proposition.

\begin{proposition}
For every $\lambda\in\mathbb{R}$, the operator $L-\lambda \iota:H^{2}(\gamma)\to L^{2}(\gamma)$ defined as $(L-\lambda\iota)(J)=L(J)-\lambda\iota(J)$ is Fredholm of index $0$. The spectrum of $L$ consists of an increasing sequence of eigenvalues $\lambda_{1}\leq \lambda_{2}\leq...$ with $\lim_{n\to\infty}\lambda_{n}=+\infty$ (i.e. $L-\lambda\iota$ has nontrivial kernel if and only if $\lambda\in\{\lambda_{n}\}_{n\in\mathbb{N}}$ and has a continuous inverse $(L-\lambda\iota)^{-1}:L^{2}(\gamma)\to H^{2}(\gamma)$ otherwise). In addition, there exists sequence $\{J_{n}\}_{n\in\mathbb{N}}$ in $H^{2}(\gamma)$ such that $\{\iota(J_{n})\}_{n\in\mathbb{N}}$ is an orthonormal basis of $L^{2}(\gamma)$ and $L(J_{n})=\lambda_{n}\iota(J_{n})$ for each $n\in\mathbb{N}$. Therefore, we have the following min-max characterization of the eigenvalues of $L$
\begin{equation*}
    \lambda_{n}=\min_{W}\max_{J\in W\setminus\{0\}}\frac{\langle L(J),\iota(J)\rangle}{\langle \iota(J),\iota(J)\rangle}
\end{equation*}
where the minimum is taken over all $n$-dimensional subspaces $W\subseteq H^{2}(\gamma)$.
\end{proposition}

\begin{proof}
    Let $\lambda\in\mathbb{R}$. Then if $J\in H^{2}_{0}(\gamma)$ and $\tilde{J}\in S^{2}(\gamma)$,
    \begin{equation*}
        (L-\lambda\iota)(J+\tilde{J})=((L_{E}(J_{E})-\lambda J_{E})_{E}+(L_{E}(\tilde{J}_{E}^{\perp})-\lambda \tilde{J}_{E}^{\perp})_{E},(B_{v}(J+\tilde{J})-\lambda \tilde{J}(v))_{v})
    \end{equation*}
where $L_{E}:H^{2}_{0}(E)\to L^{2}(E)$ is (a constant multiple of) the Jacobi operator along $\gamma_{E}$ given by $J\mapsto-\frac{n(E)}{l(E)}(\ddot{J}+R(\dot{\gamma},J)\dot{\gamma})$. We know that each $L_{E}$ is elliptic, and therefore $L_{E}-\lambda$ is Fredholm of index $0$ for every $\lambda\in\mathbb{R}$. %(as $L-\lambda$ is elliptic for every $\lambda$)%.
This implies that the product operator $\tilde{L}:H^{2}_{0}(\gamma)\to L^{2}(\mathscr{E})$, $\tilde{L}=(L_{E})_{E}$ verifies that $\tilde{L}-\lambda$ is Fredholm of index $0$ for every $\lambda\in\mathbb{R}$. Thus the fact that $L-\lambda\iota$ is always Fredholm of index $0$ can be deduced from the following lemma:

\begin{lemma}\label{Lemma Fredholm}
Let $E_{1}$, $E_{2}$, $\overline{E}_{1}$, $\overline{E}_{2}$ be Banach spaces with $\dim(E_{2})=\dim(\overline{E}_{2})<\infty$. Let $L:E_{1}\oplus E_{2}\to\overline{E}_{1}\oplus\overline{E}_{2}$ be a continuous linear map, and write $L(e_{1}.e_{2})=(L_{11}(e_{1})+L_{21}(e_{2}),L_{12}(e_{1})+L_{22}(e_{2}))$ with $L_{ij}:E_{i}\to\overline{E}_{j}$. Assume $L_{11}$ is Fredholm of index $0$. Then $L$ is Fredholm of index $0$.
\end{lemma}

\begin{proof}[Proof of Lemma \ref{Lemma Fredholm}]
    Let $\tilde{L}:E_{1}\oplus E_{2}\to \overline{E}_{1}\oplus \overline{E}_{2}$ be the operator $\tilde{L}(e_{1},e_{2})=(L_{11}(e_{1}),0)$. Because $L_{11}$ is Fredholm of index $0$ and $\dim(E_{2})=\dim(\overline{E}_{2})$, we see that $\tilde{L}$ is also Fredholm of index $0$. As $L=\tilde{L}+F$ with $F(e_{1},e_{2})=(L_{21}(e_{2}),L_{12}(e_{1})+L_{22}(e_{2}))$ compact because of the finite dimensionality of $E_{2}$ and $\overline{E}_{2}$, by \cite[Theorem~12-5.13]{Tsoy} we deduce that $L$ is also Fredholm of index $0$.
\end{proof}

Now we are going to show that the quadratic form $\Hess_{\gamma}\length_{g}:H^{2}(\gamma)\times H^{2}(\gamma)\to\mathbb{R}$ is bounded from below. We know
\begin{equation*}
    \Hess_{\gamma}\length_{g}(J,\tilde{J})=\sum_{E\in\mathscr{E}}\int_{E}\langle L_{E}(J_{E}^{\perp})(t),\tilde{J}_{E}^{\perp}(t)\rangle_{g}dt+\sum_{v\in\mathscr{V}}\langle B_{v}(J),\tilde{J}(v)\rangle_{g}.
\end{equation*}
Denote by $C:H^{2}(\gamma)\times H^{2}(\gamma)\to\mathbb{R}$ the form $C(J,\tilde{J})=\sum_{v\in\mathscr{V}}\langle B_{v}(J),\tilde{J}(v)\rangle_{g}$. $C$ is symmetric because so are $\Hess_{\gamma}\length_{g}$ and $L_{E}$ for each $E\in\mathscr{E}$. If we endow $S^{2}(\gamma)$ with the inner product
\begin{equation*}
    \langle J,\tilde{J}\rangle=\sum_{v\in\mathscr{V}}\langle J(v),\tilde{J}(v)\rangle_{g}
\end{equation*}
then as $\dim(S^{2}(\gamma))<\infty$, we can see that there exists some constant $\alpha>0$ such that
\begin{equation}\label{form bound 1}
    |C(J,J)|\leq\alpha\sum_{v\in\mathscr{V}}\langle J(v),J(v)\rangle_{g}
\end{equation}
for every $J\in S^{2}(\gamma)$. But then as $C$ vanishes on $H^{2}_{0}(\gamma)$, by its bilinearity and symmetry we can see that in fact (\ref{form bound 1}) is valid for every $J\in H^{2}(\gamma)$.

On the other hand, using that each $L_{E}$ is elliptic, for each $E\in\mathscr{E}$ there exists $\beta_{E}\in\mathbb{R}$ such that
\begin{equation}\label{form bound 2}
    \int_{E}\langle L_{E}(J^{\perp}_{E})(t),J_{E}^{\perp}(t)\rangle_{g}dt\geq\beta_{E}\int_{E}\langle J_{E}^{\perp}(t),J_{E}^{\perp}(t)\rangle_{g}dt.
\end{equation}
Thus if $\beta=\min\{\beta_{E}:E\in\mathscr{E}\}$ and $\gamma=\min\{\beta,-\alpha\}$, from (\ref{form bound 1}) and (\ref{form bound 2}) we deduce
\begin{equation*}
    \Hess_{\gamma}\length_{g}(J,J)\geq\beta\sum_{E\in\mathscr{E}}\int_{E}\langle J_{E}^{\perp}(t),J_{E}^{\perp}(t)\rangle_{g}dt-\alpha \sum_{v\in\mathscr{V}}\langle J(v),J(v)\rangle_{g}\geq \gamma\langle \iota(J),\iota(J)\rangle
\end{equation*}
which considering that $\Hess_{\gamma}\length_{g}(J,J)=\langle L(J),\iota(J)\rangle$ implies that for every $\lambda\in\mathbb{R}$ it holds
\begin{equation*}
    \langle (L+\lambda\iota)(J),\iota(J)\rangle\geq (\lambda+\gamma)\langle\iota(J),\iota(J)\rangle
\end{equation*}
and in particular if $\lambda>-\gamma$ implies that $L+\lambda\iota$  is a monomorphism. Because we also know that these operators are Fredholm of index $0$, by the Open Mapping Theorem we conclude that $L+\lambda\iota:H^{2}(\gamma)\to L^{2}(\gamma)$ is a continuous linear isomorphism for every $\lambda>-\gamma$.

Fix $\lambda>-\gamma$. We will now show that $\iota\circ(L+\lambda\iota)^{-1}:L^{2}(\gamma)\to L^{2}(\gamma)$ is compact. Let $(X_{n})_{n\in\mathbb{N}}$ be a bounded sequence in $L^{2}(\gamma)$ and define $(J^{n},\tilde{J}^{n})=(L+\lambda\iota)^{-1}(X_{n})$ with $J^{n}\in H^{2}_{0}(\gamma)$ and $\tilde{J}^{n}\in S^{2}(\gamma)$. As $(L+\lambda\iota)^{-1}$ is bounded, $(J^{n},\tilde{J}^{n})$ is a bounded sequence in $H^{2}(\gamma)$. Therefore, for each $E\in\mathscr{E}$ the sequence of normal vector fields $(J^{n}_{E})_{n\in\mathbb{N}}$ along $\gamma_{E}$ is bounded in $H^{2}_{0}(E)$ and therefore in $H^{1}_{0}(E)$. Hence, by the Rellich-Kondrachov Compactness Theorem we can find a subsequence $(n_{k})_{k\in\mathbb{N}}$ such that $(J^{n_{k}}_{E})_{k\in\mathbb{N}}$ converges in $L^{2}(E)$ for every $E\in\mathscr{E}$. On the other hand, using that $S^{2}(\gamma)$ is finite dimensional, we can extract a further subsequence $(n_{k_{l}})_{l}$ to have the additional property that $(\tilde{J}^{n_{k_{l}}})_{l\in\mathbb{N}}$ converges in $S^{2}(\gamma)$. This implies that the sequence of general term $\iota\circ(L+\lambda\iota)^{-1}(X_{n_{k_{l}}})=\iota(J^{n_{k_{l}}},\tilde{J}^{n_{k_{l}}})$ converges in $L^{2}(\gamma)$, and this completes the proof that $\iota\circ(L+\lambda\iota)^{-1}$ is compact.

The symmetry of $\Hess_{\gamma}\length_{g}(J,\tilde{J})=\langle L(J),\iota(\tilde{J})\rangle$ implies that $\iota\circ(L+\lambda\iota)^{-1}$ is self-adjoint, which together with its compactness implies the existence of an orthonormal basis $\{X_{n}\}_{n\in\mathbb{N}}$ of $L^{2}(\gamma)$ such that $\iota\circ(L+\lambda\iota)^{-1}X_{n}=\delta_{n}X_{n}$ for some decreasing sequence $\delta_{n}\to 0^{+}$ (because by our choice of $\lambda$, $\iota\circ(L+\lambda\iota)^{-1}\geq 0$). But we claim that $X\in L^{2}(\gamma)$ is an eigenvector of $\iota\circ(L+\lambda\iota)^{-1}$ of eigenvalue $\delta\in\mathbb{R}$ if and only if $X=\iota(J)$ for some $J\in H^{2}(\gamma)$ such that $L(J)=(\delta^{-1}-\lambda)\iota(J)$. This is because $\iota\circ(L+\lambda\iota)^{-1}(X)=\delta X$ if and only if there exists $J\in H^{2}(\gamma)$ with $\iota(J)=X$ which verifies any of the the following equivalent conditions:
\begin{align*}
    \iota\circ(L+\lambda\iota)^{-1}\circ\iota(J)& =\delta\iota(J)\\
    (L+\lambda\iota)^{-1}\circ\iota(J) & = \delta J\\
    \iota(J) & =\delta(L+\lambda\iota)(J)\\
    L(J) & =(\delta^{-1}-\lambda)\iota(J).
\end{align*}
From the previous, we conclude that if $\lambda_{n}:=\delta_{n}^{-1}-\lambda$ then $\spectrum(L)=\{\lambda_{n}\}_{n}$, $\lim_{n\to\infty}\lambda_{n}=+\infty$ and $L(J_{n})=\lambda_{n}\iota(J_{n})$. This implies the min-max theorem for the eigenvalues holds for $L$, which completes the proof.
\end{proof}

\section{Some auxiliary results}\label{two technical propositions}

\begin{proposition}\label{Perturbation}
Let $g:I^{N}\to\mathcal{M}^{q}$ be a smooth embedding, $N\in\mathbb{N}$, $I=(-1,1)$. If $q\geq N+3$, there exists an arbitrarily small perturbation in the $C^{\infty}$ topology $g':I^{N}\to\mathcal{M}^{q}$ such that there is a full measure subset $\mathcal{A}\subseteq I^{N}$ with the following properties: for any $p\in\mathbb{N}$ and any $t\in\mathcal{A}$, the function $s\mapsto\omega^{1}_{p}(g'(s))$ is differentiable at $t$, and there exists a (possibly disconnected) weighted multigraph $\Gamma$ and a stationary geodesic network $\gamma_{p}=\gamma_{p}(t):\Gamma\to (M,g'(t))$ such that the following two conditions hold
\begin{enumerate}
    \item $\omega_{p}^{1}(g'(t))=L_{g'(t)}(\gamma_{p}(t))$.
    \item $\frac{\partial}{\partial v}(\omega^{1}_{p}\circ g')\big|_{s=t}=\frac{1}{2}\int_{\gamma_{p}(t)}\trace_{\gamma_{p}(t),g'(t)}\frac{\partial g'}{\partial v}(t)\dL_{g'(t)}$.
\end{enumerate}
\end{proposition}

To prove the proposition, we will need to have a condition for a sequence of embedded stationary geodesic nets $\gamma_{n}:\Gamma\to (M,g_{n})$ converging to some $\gamma:\Gamma\to (M,g)$ that guarantees that $\gamma$ is also embedded. The condition we will work with can be expressed as a collection of lower and upper bounds of certain functionals defined for pairs $(g,\gamma)$ where $\gamma$ is stationary with respect to $g$. We proceed to describe those functionals.

The first one is
\begin{equation*}
    F_{1}(g,\gamma)=\min\{|\dot{\gamma}_{E}(t)|_{g}:t\in E,E\in\mathscr{E}\}.
\end{equation*}
A lower bound for this functional will imply that the limit net is an immersion along each edge.

Then we have a family of functionals $F_{2}^{(E_{1},i_{1}),(E_{2},i_{2})}$ defined for each pair \linebreak $((E_{1},i_{1}),(E_{2},i_{2}))\in(\mathscr{E}\times\{0,1\})^{2}$ such that $\pi_{E_{1}}(i_{1})=\pi_{E_{2}}(i_{2})$ (see Section \ref{geodesic networks} for the notation) as follows

\begin{equation*}
    F_{2}^{(E_{1},i_{1}),(E_{2},i_{2})}(g,\gamma)=(-1)^{i_{1}+i_{2}}\frac{\langle\dot{\gamma}_{E_{1}}(i_{1}),\dot{\gamma}_{E_{2}}(i_{2})\rangle_{g}}{|\dot{\gamma}_{E_{1}}(i_{1})|_{g}|\dot{\gamma}_{E_{2}}(i_{2})|_{g}}.
\end{equation*}
Notice that $(-1)^{i_{j}}\frac{\dot{\gamma}_{E_{j}}(i_{j})}{|\dot{\gamma}_{E_{j}}(i_{j})|_{g}}$ is the unit inward tangent vector to $\gamma$ at $v=\pi_{E_{j}}(i_{j})$ along $E_{j}$, $j=1,2$ (and observe that in case $E$ is a loop at $v$, there are two inward tangent vectors to $\gamma$ along $E$ at $v$ represented by the pairs $(E,0)$ and $(E,1)$). The condition $F_{2}^{(E_{1},i_{1}),(E_{2},i_{2})}(g_{n},\gamma_{n})\leq 1-\delta$ for some $\delta>0$ and for every possible choice $(E_{1},i_{1})\neq (E_{2},i_{2})$ with $\pi_{E_{1}}(i_{1})=\pi_{E_{2}}(i_{2})$ implies that the limit $(g,\gamma)$ has the property that given $v\in\mathscr{V}$, there exists an open neighborhood $U_{v}$ of $v$ in $\Gamma$ such that $\gamma:U_{v}\to \gamma(U_{v})$ is a homeomorphism. Explicitly,
\begin{equation*}
    U_{v}=\bigcup_{(E,i):\pi_{E}(i)=v}\{t\in E:|t-i|<\min\{\frac{\inj(g)}{\length_{g}(\gamma_{E})},\frac{1}{2}\}\}
\end{equation*}
where $\inj:\mathcal{M}^{q}\to\mathbb{R}_{>0}$, $g\mapsto \inj(g)$ is a continuous choice of the injectivity radius for each $C^q$ Riemannian metric $g$. This is because if we consider $U_{v}$ as a graph obtained by gluing at $v$ one edge for each pair $(E,i)\in\mathscr{E}\times\{0,1\}$ such that $\pi_{E}(i)=v$, this graph is mapped by $\gamma$ into a geodesic ball centered at $\gamma(v)$ of radius $\inj(g)$ and the image of each incoming edge at $v$ has a different inward tangent vector at $\gamma(v)$.

To ensure injectivity along the edges, we define for each edge $E\in\mathscr{E}$ a function
\begin{equation*}
    d^{E}_{(g,\gamma)}(t)=\min\{d_{g}(\gamma(t),\gamma(s)):s\in E, |t-s|\geq\frac{\inj(g)}{\length_{g}(\gamma_{E})}\}.
\end{equation*}
In case $\pi_{E}(0)=\pi_{E}(1)$, the distance $|s-t|$ between two points $s,t\in E$ is measured with respect of the length of $S^{1}=E/0\sim 1$.

To ensure that the images of different edges under $\gamma$ do not overlap, we define for each pair $E,E'\in\mathscr{E}$, $E\neq E'$ a function $d^{E,E'}_{(g,\gamma)}:E\to\mathbb{R}_{\geq 0}$ as
\begin{align*}
    d^{E,E'}_{(g,\gamma)}(t)= & \min\{d_{g}(\gamma(t),\gamma(s)):s\in E', |s-i|\geq\frac{\inj(g)}{\length_{g}(\gamma_{E'})}
    \text{ for each }i\in\{0,1\}\\
    & \text{ s.t. }\exists j\in\{0,1\}\text{ with }\pi_{E'}(i)=\pi_{E}(j)\text{ and }|t-j|\leq\frac{\inj(g)}{\length_{g}(\gamma_{E})}\}.
\end{align*}

Let us fix an auxiliary embedding $\psi:M\to\mathbb{R}^{l}$ and identify from now on our manifold $M$ with the submanifold $\psi(M)\subseteq\mathbb{R}^{l}$. Given a multigraph $\Gamma$ and a continuous map $\gamma:\Gamma\to M$ which is $C^3$ when restricted to each edge, we can consider

\begin{equation*}
    \Vert \gamma\Vert_{3}=\Vert \gamma\Vert_{0}+\Vert \dot{\gamma}\Vert_{0}+\Vert\Ddot{\gamma}\Vert_{0}+\Vert\dddot \gamma\Vert_{0}
\end{equation*}
where given a collection $u=(u_{E})_{E\in\mathscr{E}}$ of continuous functions along the edges of $\Gamma$, we define
\begin{equation*}
    \Vert u\Vert_{0}=\max\{|u_{E}(t)|:t\in E,E\in\mathscr{E}\}
\end{equation*}
being $|\cdot|$ the Euclidean norm in $\mathbb{R}^{l}$. We have the following compactness result.

\begin{lemma}\label{Compactness sgn}
Let $(g_{n})_{n\in\mathbb{N}}$ be a sequence of $C^{3}$ Riemannian metrics converging to some metric $g\in\mathcal{M}^{3}$. Let $\gamma_{n}:\Gamma\to (M,g_{n})$ be a sequence of stationary geodesic networks. Assume $\Vert \gamma_{n}\Vert_{3}\leq M$ for some $M\in\mathbb{R}_{>0}$. Then there exists a subsequence $(\gamma_{n_{k}})_{k}$ and $\gamma\in\Omega(\Gamma,M)$ such that $\lim_{k\to\infty}\gamma_{n_{k}}=\gamma$ in $\Omega(\Gamma,M)$ and $\gamma:\Gamma\to(M,g)$ is stationary.
\end{lemma}

\begin{proof}
    The Arzela-Ascoli Theorem gives a subsequence $\gamma_{n_{k}}\to \gamma$ in $\Omega(\Gamma,M)$. The fact that $\gamma$ is stationary with respect to $g$ comes from the continuity of the operator $H$ defined in \cite{Staffa} (which plays the role of the mean curvature operator on minimal surfaces) which vanishes in a pair $(g,[\gamma])$ if and only if $\gamma$ is stationary with respect to $g$.
\end{proof}
 
We will also need the following two lemmas.

\begin{lemma}\label{lcountable}
Let $F:\mathbb{R}^{n}\to\mathbb{N} $ be a function. Then there exists $m\in\mathbb{N}$ and a basis $\{v_{1},...,v_{n}\}$ of $\mathbb{R}^{n}$ such that $F(v_{i})=m$ for all $1\leq i\leq n$.
\end{lemma}

\begin{proof}
Observe that $\mathbb{R}^{n}=\bigcup_{m\in\mathbb{N}}F^{-1}(m)$ and therefore $\mathbb{R}^{n}=\bigcup_{m\in\mathbb{N}}\langle F^{-1}(m)\rangle$ where given $A\subseteq\mathbb{R}^{n}$ we denote $\langle A\rangle$ the subspace spanned by $A$. If $F^{-1}(m)$ did not contain a basis of $\mathbb{R}^{n}$ for every $m\in\mathbb{N}$, $\langle F^{-1}(m)\rangle$ would a proper subspace for every $m$. Therefore, $\mathbb{R}^{n}$ would be a countable union of closed subspaces with empty interior, which leads to a contradiction due to the Baire Category Theorem.
\end{proof}

\begin{lemma}\label{Diff}
Let $\gamma:(-1,1)^{N}\to\Omega(\Gamma,M)$ and $g:(-1,1)^{N}\to\mathcal{M}^{q}$ be smooth maps. Assume that $\gamma(s)$ is stationary with respect to $g(s)$ for every $s\in(-1,1)^{N}$. Then for every $t\in(-1,1)^{N}$ and every $v\in\mathbb{R}^{N}$
\begin{equation*}
    \frac{\partial}{\partial v}\big|_{s=t}\length(g(s),\gamma(s))=\frac{1}{2}\int_{\gamma(t)}\trace_{\gamma(t),g(t)}\frac{\partial g}{\partial v}(t)\dL_{g(t)}.
\end{equation*}
\end{lemma}

\begin{proof}
Using that the length functional is a smooth function $\length:\mathcal{M}^{q}\times\Omega(\Gamma,M)\to\mathbb{R}$ and the chain rule, we get

\begin{align*}
    \frac{\partial}{\partial v}\big|_{s=t}\length(g(s),\gamma(s)) & = D\length_{(g(t),\gamma(t))}(D(g\times \gamma)_{t}(v))\\
    & = D\length_{(g(t),\gamma(t))}(\frac{\partial g}{\partial v}(t),\frac{\partial \gamma}{\partial v}(t))\\
    & =D_{1}\length_{(g(t),\gamma(t))}(\frac{\partial g}{\partial v}(t))+D_{2}\length_{g(t),\gamma(t))}(\frac{\partial \gamma}{\partial v}(t))\\
    & =D_{1}\length_{(g(t),\gamma(t))}(\frac{\partial g}{\partial v}(t)).
\end{align*}
The second term in the penultimate equation vanishes because $\gamma(t)$ is stationary with respect to $g(t)$. Hence
\begin{align*}
    \frac{\partial}{\partial v}\big|_{s=t}\length(g(s),\gamma(s)) & =\frac{d}{ds}\big|_{s=0}\length(g(t+sv),\gamma(t))\\
    & = \frac{d}{ds}\big|_{s=0}\sum_{E\in\mathscr{E}}n(E)\int_{E}\sqrt{g_{t+sv}(\dot{\gamma}_{t}(u),\dot{\gamma}_{t}(u))}du\\
    & = \sum_{E\in\mathscr{E}}n(E)\int_{E}\frac{d}{ds}\big|_{s=0}\sqrt{g_{t+sv}(\dot{\gamma}_{t}(u),\dot{\gamma}_{t}(u))}du\\
    & =\sum_{E\in\mathscr{E}}n(E)\int_{E}\frac{\frac{\partial g}{\partial v}(t)(\dot{\gamma}_{t}(u),\dot{\gamma}_{t}(u))}{2\sqrt{g_{t}(\dot{\gamma}_{t}(u),\dot{\gamma}_{t}(u))}}du\\
    & =\frac{1}{2}\sum_{E\in\mathscr{E}}n(E)\int_{E}\frac{\frac{\partial g}{\partial v}(t)(\dot{\gamma}_{t}(u),\dot{\gamma}_{t}(u))}{g_{t}(\dot{\gamma}_{t}(u),\dot{\gamma}_{t}(u))}\sqrt{g_{t}(\dot{\gamma}_{t}(u),\dot{\gamma}_{t}(u))}du\\
    & =\frac{1}{2}\sum_{E\in\mathscr{E}}n(E)\int_{\gamma(t)_{E}}\trace_{\gamma(t),g(t)}\frac{\partial g}{\partial v}(t)\dL_{g(t)}\\
    & =\frac{1}{2}\int_{\gamma(t)}\trace_{\gamma(t),g(t)}\frac{\partial g}{\partial v}(t)\dL_{g(t)}.
\end{align*}
\end{proof}

\begin{proof}[Proof of Proposition \ref{Perturbation}]
Notice that it suffices to show that for each $p\in\mathbb{N}$, there exists a full measure subset $\mathcal{A}(p)\subseteq I^{N}$ where \textit{(1)} and \textit{(2)} hold, because in that case $\mathcal{A}=\bigcap_{p\in\mathbb{N}}\mathcal{A}(p)$ will have the desired property. Therefore we will assume $p\in\mathbb{N}$ is fixed.

Let $g:I^{N}\to\mathcal{M}^{q}$ be a smooth embedding. Let $\{\Gamma_{i}\}_{i\geq 1}$ be a sequence enumerating the countable collection of all good weighted multigraphs. For each $i\geq 1$, let $\mathcal{S}^{q}(\Gamma_{i})$ be the space of pairs $(g,[\gamma])$ where $g\in\mathcal{M}^{q}$, $\gamma:\Gamma_{i}\to (M,g)$ is an embedded stationary geodesic net and $[\gamma]$ denotes its class modulo reparametrization as defined in \cite{Staffa} for connected multigraphs with at least three incoming edges at each vertex and in \cite{White} for embedded closed geodesics. By the structure theorems proved in \cite{Staffa} and \cite{White}, each $\mathcal{S}^{q}(\Gamma_{i})$ is a second countable Banach manifold and the projection map $\Pi_{i}:\mathcal{S}^{q}(\Gamma_{i})\to\mathcal{M}^{q}$ mapping $(g,[\gamma])\mapsto g$ is Fredholm of index $0$. A pair $(g,[\gamma])\in\mathcal{S}_{q}(\Gamma_{i})$ is a critical point of $\Pi_{i}$ if and only if $\gamma$ admits a nontrivial Jacobi field with respect to the metric $g$.

By Smale's transversality theorem, we can perturb $g:I^{N}\to\mathcal{M}^{q}$ slightly in the $C^{\infty}$ topology to a $C^{\infty}$ embedding $g':I^{N}\to\mathcal{M}^{q}$ which is transversal to $\Pi_{i}:\mathcal{S}^{q}(\Gamma_{i})\to\mathcal{M}^{q}$ for every $i\in\mathbb{N}$. Transversality implies that $M_{i}=\Pi_{i}^{-1}(g'(I^{N}))$ is an $N$-dimensional embedded submanifold of $\mathcal{S}^{q}(\Gamma_{i})$ for every $i\in\mathbb{N}$. Let $\pi_{i}=(g')^{-1}\circ\Pi_{i}\big|_{M_{i}}:M_{i}\to I^{N}$. Let $\tilde{\mathcal{A}}_{i}\subseteq I^{n}$ be the set of regular values of $\pi_{i}$, which is a set of full measure by Sard's theorem. Let $\tilde{\mathcal{A}}_{0}\subseteq I^{N}$ be the set of points for which the Lipschitz function $s\mapsto\omega^{1}_{p}(g'(s))$ is differentiable. Observe that $\tilde{\mathcal{A}}_{0}$ has full measure by Rademacher's theorem. Therefore, $\tilde{\mathcal{A}}=\bigcap_{i\geq 0}\tilde{\mathcal{A}}_{i}$ is a full measure subset of $I^{N}$. Notice that by transversality, if $t\in\tilde{\mathcal{A}}$ then $g'(t)$ is a bumpy metric, i.e. all embedded stationary geodesic nets with respect to $g'(t)$ and with domain a good weighted multigraph are nondegenerate; and also the map $s\mapsto\omega_{p}^{1}(g'(s))$ is differentiable at $s=t$.

Given a weighted multigraph $\Gamma=\bigcup_{i=1}^{P}\Gamma_{i}$ whose connected components $\Gamma_{i}$ are good and a natural number $M\in\mathbb{N}$, we define $\mathcal{B}_{\Gamma,M}$ as the set of all $t\in I^{N}$ such that there exists a stationary geodesic network $\gamma:\Gamma\to (M,g'(t))$ verifying

\begin{enumerate}
    \item $\gamma_{i}=\gamma|_{\Gamma_{i}}$ is an embedding for each $1\leq i\leq P$.
    \item $\Vert \gamma_{i}\Vert_{3}\leq M$ for every $1\leq i\leq P$.
    \item $F_{1}(g'(t),\gamma_{i})\geq\frac{1}{M}$ for every $1\leq i\leq P$.
    \item $F_{2}^{(E_{1},i_{1}),(E_{2},i_{2})}(g'(t),\gamma_{i})\leq 1-\frac{1}{M}$ for every $1\leq i\leq P$ and every pair $(E_{1},i_{1})\neq(E_{2},i_{2})$  in $\mathscr{E}_{i}\times\{0,1\}$ such that $\pi_{E_{1}}(i_{1})=\pi_{E_{2}}(i_{2})$.
    \item $d^{E}_{(g'(t),\gamma_{i})}(s)\geq\frac{1}{M}$ for every $1\leq i\leq P$, $E\in\mathscr{E}_{i}$ and $s\in E$.
    \item $d^{E,E'}_{(g'(t),\gamma_{i})}(s)\geq\frac{1}{M}$ for every $1\leq i\leq P$, $E\neq E'\in\mathscr{E}_{i}$ and $s\in E$.
    \item $\omega_{p}^{1}(g'(t))=\length_{g'(t)}(\gamma)$.
\end{enumerate}
where $\mathscr{E}_{i}$ denotes the set of edges of $\Gamma_{i}$. Observe that $I^{N}=\bigcup_{\Gamma,M}\mathcal{B}_{\Gamma,M}$ because of (\ref{Almgren widths}) and Remark \ref{unionofregular}. We claim that each $\mathcal{B}_{\Gamma,M}\subseteq I^{N}$ is closed.

Indeed, suppose we have a sequence $\{t_{j}\}_{j\in\mathbb{N}}\subseteq\mathcal{B}_{\Gamma,M}$ converging to some $t\in I^{N}$. Let $\gamma^{j}$ be the stationary geodesic network corresponding to $g'(t_{j})$ and verifying properties (1) to (7) above. By property (2) and Lemma (\ref{Compactness sgn}), passing to a subsequence we have that if $\gamma^{j}_{i}=\gamma^{j}\big|_{\Gamma_{i}}$ then there exists $\gamma_{i}:\Gamma_{i}\to M$ such that $\lim_{j\to\infty}\gamma_{i}^{j}=\gamma_{i}$ in $\Omega(\Gamma_{i},M)$ and $\gamma_{i}$ is stationary with respect to $g'(t)$ for each $1\leq i\leq P$. Observe also that if $\gamma=\bigcup_{j}\gamma_{i}$
\begin{equation*}
    L_{g'(t)}(\gamma)=\lim_{j\to\infty}L_{g'(t_{j})}(\gamma^{j})=\lim_{j\to\infty}\omega_{p}(t_{j})=\omega_{p}(t).
\end{equation*}

Properties (2) to (6) are preserved when we take the limit of the sequence $\gamma^{j}$, so it suffices to show that $\gamma|_{\Gamma_{i}}$ is embedded for each $1\leq i\leq P$. Fix such $i$. Properties (3), (4) and (5) imply that $\gamma_{i}$ is injective along the edges and property (6) combined with property (4) imply that the images of different edges do not intersect (except at the common vertices).

As each $\mathcal{B}_{\Gamma,M}$ is closed, they are measurable and therefore so are the sets $\tilde{\mathcal{A}}_{\Gamma,M}=\tilde{\mathcal{A}}\cap\mathcal{B}_{\Gamma,M}$ (whose union is $\tilde{\mathcal{A}}$). Let $\mathcal{\mathcal{A}}_{\Gamma,M}$ be the set of points $t\in\tilde{\mathcal{A}}_{\Gamma,M}$ where the Lebesgue density of $\tilde{\mathcal{A}}_{\Gamma,M}$ at $t$ is $1$. By the Lebesgue Differentiation Theorem, $\tilde{\mathcal{A}}_{\Gamma,M}\setminus\mathcal{A}_{\Gamma,M}$ has Lebesgue measure $0$ for each pair $(\Gamma,M)$. Let us define $\mathcal{A}=\bigcup_{\Gamma,M}\mathcal{A}_{\Gamma,M}$, observe that as $\mathcal{\tilde{A}}\setminus\mathcal{A}$ has measure $0$, $\mathcal{A}\subseteq I^{N}$ has full measure.

Fix $t\in\mathcal{A}$. Let $(\Gamma,M)$ be such that $t\in\mathcal{A}_{\Gamma,M}$. As the density of $\tilde{\mathcal{A}}_{\Gamma,M}$ at $t$ is $1$,  given $v\in\mathbb{R}^{N}$ with $|v|=1$ we can find a sequence $\{t_{m}(v)\}_{m\in\mathbb{N}}\subseteq\tilde{\mathcal{A}}_{\Gamma,M}$ such that $\lim_{m\to\infty}t_{m}(v)=t$ and $\lim_{m\to\infty}\frac{t-t_{m}(v)}{|t-t_{m}(v)|}=v$. Denoting $\omega_{p}(s)=\omega_{p}^{1}(g'(s))$, using that $\omega_{p}$ is a Lipschitz function we can see that
\begin{equation}\label{w1}
    \lim_{m\to\infty}\frac{\omega_{p}(t_{m}(v))-\omega_{p}(t)}{|t-t_{m}|}=\frac{\partial}{\partial v}\omega_{p}(t).
\end{equation}

As $t_{m}(v)\in\tilde{\mathcal{A}}_{\Gamma,M}$, for each $m\in\mathbb{N}$ there exists a stationary geodesic network $\gamma_{m}:\Gamma\to M$ with respect to $g'(t_{m}(v))$ such that
\begin{equation*}
    \omega_{p}(t_{m}(v))=\omega^{1}_{p}(g'(t_{m}(v)))=\length_{g'(t_{m}(v))}(\gamma_{m})
\end{equation*}
and properties (1) to (6) above hold. By the reasoning used to prove that the $\mathcal{B}_{\Gamma,M}$ are closed, we can construct a stationary geodesic net $\gamma:\Gamma\to (M,g'(t))$ which is embedded when restricted to each connected component $\Gamma_{i}$ of $\Gamma$, is the limit of (a subsequence of) the $\gamma_{m}$'s in the $C^{2}$ topology and realizes the width $\omega_{p}^{1}(g'(t))$. Hence from (\ref{w1}) we get

\begin{equation*}
    \frac{\partial}{\partial v}\omega_{p}(t) =\lim_{m\to\infty}\frac{\length_{g'(t_{m})}(\gamma_{m})-\length_{g'(t)}(\gamma)}{|t-t_{m}|}.
\end{equation*}

As $\gamma|_{\Gamma_{i}}$ is an embedded stationary geodesic net with respect to $g'(t)$ for each $1\leq i\leq P$ and $g'(t)$ is bumpy, $\Pi_{i}:\mathcal{S}^{q}(\Gamma_{i})\to\mathcal{M}^{q}$ is a diffeomorphism from a neighborhood $U_{i}$ of $(g'(t),[\gamma_{i}])$ to a neighborhood $W_{i}=\Pi_{i}(U)$ of $g'(t)$. Denote $\Xi_{i}$ its inverse. As there exists $m_{0}\in\mathbb{N}$ such that $g'(t_{m})\in W=\bigcap_{i=1}^{P}W_{i}$ and $[\gamma_{m}\big|_{\Gamma_{i}}]\in U_{i}$ for every $m\geq m_{0}$, we deduce that $[\gamma_{m}\big|_{\Gamma_{i}}]=\Xi_{i}(g'(t_{m}(v)))$ for each $m\geq m_{0}$ and each $1\leq i\leq P$. Let us define $\Xi:W\to\hat{\Omega}(\Gamma,M)$ as $\Xi(g)=h$ where $h|_{\Gamma_{i}}=\Xi_{i}(g)$. Thus by Lemma \ref{Diff}

\begin{align*}
    \frac{\partial}{\partial v}\omega_{p}(t) & =\lim_{m\to\infty}\frac{\length_{g'(t_{m})}(\Xi(g'(t_{m})))-\length_{g'(t)}(\Xi(g'(t)))}{|t_{m}-t|}\\
    & = \frac{\partial}{\partial v}\big|_{s=t}\length(g'(s),\Xi(g'(s)))\\
    & =\frac{1}{2}\int_{\gamma_{v}}\trace_{\gamma_{v},g'(t)}\frac{\partial g'}{\partial v}(t)\dL_{g'(t)}.
\end{align*}
Where $\gamma_{v}=\Xi(g'(t))$ is the one constructed before. Observe that $\gamma_{v}$ depends on $v$ and that the previous formula holds for each $v\in\mathbb{R}^{N}$, $|v|=1$. Notice that each $\gamma_{v}$ is a stationary geodesic network with respect to $g'(t)$, and as $g'(t)$ is bumpy there are countably many possible $\gamma_{v}'s$, say $\{h_{j}\}_{j\in\mathbb{N}}$. This induces a map $F:\mathbb{R}^{N}\to\mathbb{N}$ defined as $F(0)=1$ and if $w\neq 0$ then $F(w)= j$ where $\gamma_{\frac{w}{|w|}}=h_{j}$. By Lemma \ref{lcountable} we can obtain $m\in\mathbb{N}$ and a basis $w_{1},...,w_{N}$ of $\mathbb{R}^{N}$ with the property $\gamma(w_{i})=m$ for every $1\leq i\leq N$. Therefore if we set $v_{i}:=\frac{w_{i}}{|w_{i}|}$, $v_{1},...,v_{N}$ is still a basis and by definition $\gamma_{v_{i}}=h_{m}$ for every $i$. By linearity of directional derivatives, denoting $\gamma=h_{m}$ we deduce that

\begin{equation*}
    \frac{\partial}{\partial v}\omega_{p}(t) =\frac{1}{2}\int_{\gamma}\trace_{\gamma,g'(t)}\frac{\partial g'}{\partial v}(t)\dL_{g'(t)}
\end{equation*}
for every unit $v\in\mathbb{R}^{N}$, which completes the proof.
\end{proof}

\begin{proposition}\label{Perturbation nondeg}
Let $M$ be a closed manifold and let $g$ be a $C^{q}$ Riemannian metric on $M$, $q\geq 3$. Let $\gamma_{1},...,\gamma_{k}$ be a finite collection collection of connected, embedded stationary geodesic networks on $(M,g)$ whose domains are good weighted multigraphs and let $\mathcal{U}\subseteq\mathcal{M}^{q}$ be an open neighborhood of $g$. Then there exists $g'\in\mathcal{U}$ such that $\gamma_{1},...\gamma_{k}$ are non-degenerate stationary geodesic nets with respect to $g'$.
\end{proposition}

\begin{proof}
    Following \cite[Lemma~4]{MNS}, we will consider conformal perturbations of the metric of the form $g_{\varepsilon}(x)=e^{-2\varepsilon\phi(x)}g(x)$. Let us denote $\tilde{\gamma}=\bigcup_{i=1}^{k}\gamma_{i}$, $\tilde{\Gamma}=\bigcup_{i=1}^{k}\Gamma_{i}$ (where $\gamma_{i}:\Gamma_{i}\to M$) and $\tilde{\mathscr{E}}$ the set of edges of $\tilde{\gamma}$. Notice that $\tilde{\gamma}:\tilde{\Gamma}\to M$ is a stationary geodesic network whose edges may overlap, even non-transversally. Given $E\in\mathscr{E}$, let $\reg(\tilde{\gamma}_{E})$ be the set of interior points of $\tilde{\gamma}_{E}$ which are not points of transverse intersection with any other edge $\tilde{\gamma}_{E}$. We define a finite poset
    \begin{equation*}
        \mathcal{P}=\{\bigcap_{i=1}^{l}\reg(\tilde{\gamma}_{E_{i}})\neq\emptyset:E_{1},...,E_{l}\in\tilde{\mathscr{E}},\text{ }E_{i}\neq E_{j}\text{ }\forall i\neq j\}
    \end{equation*}
    which is the collection of finite non-empty intersections of sets in $\{\reg(\tilde{\gamma}_{E}):E\in\mathscr{E}\}$, with the order given by the inclusion. Denote by $\mathcal{P}'$ the set of minimal elements in $\mathcal{P}$. Observe that if $\alpha,\alpha'$ are two different elements of $\mathcal{P}'$ then they are disjoint. Given $\alpha\in\mathcal{P}'$, write $\alpha=\bigcap_{i=1}^{l}\reg(\tilde{\gamma}_{E_{i}})$ in the unique way such that $\alpha\cap \tilde{\gamma}_{E}=\emptyset$ for every $E\in\tilde{\mathscr{E}}\setminus\{E_{1},...,E_{l}\}$. Pick $t_{\alpha}\in \alpha$ for every $\alpha\in\mathcal{P}'$, and let $\eta>0$ be such that the geodesic balls $B_{\alpha}=B(p_{\alpha},\eta)$ verify
    \begin{itemize}
        \item $B_{\alpha}\cap B_{\alpha'}=\emptyset$ if $\alpha\neq \alpha'$.
        \item $B_{\alpha}\cap \gamma_{E}=\emptyset$ if $E\notin\{E_{1},...,E_{l}\}$.
        \item $B_{\alpha}\cap \gamma_{E_{i}}\subseteq\alpha$ for every $i=1,...,l$.
        \item There exists a diffeomorphism $\rho_{\alpha}:B_{\alpha}\to\mathbb{R}^{n}$ such that $\rho_{\alpha}(\gamma_{E_{i}}\cap B_{\alpha})=\rho_{\alpha}(\alpha\cap B_{\alpha})=\{(t,0,0,...,0):t\in\mathbb{R}\}$ for each $i=1,...,l$.
    \end{itemize}
    Denote $B'_{\alpha}=B(p_{\alpha},\frac{\eta}{2})$. Observe that for each $E\in\tilde{\mathscr{E}}$ there exists at least one $\alpha\in\mathcal{P}'$ such that $\alpha\subseteq\tilde{\gamma}_{E}$. Choose such an $\alpha$ for each $E\in\mathscr{E}$ and denote $B_{E}=B_{\alpha}$ and $B'_{E}=B'_{\alpha}$. We can now proceed to define the function $\phi$ which will induce the one-parameter family of metrics $g_{\varepsilon}(x)=e^{-2\varepsilon\phi(x)}$ mentioned before.
    
    For each $\alpha\in\mathcal{P}'$, let $\psi_{\alpha}:M\to\mathbb{R}$ be a smooth function with $0\leq\psi_{\alpha}\leq 1$, $\support(\psi_{\alpha})\subseteq B_{\alpha}$ and $\psi_{\alpha}\equiv 1$ in $B'_{\alpha}$. Let $f_{\alpha}:B_{\alpha}\to\mathbb{R}$ be given in local coordinates under the chart $(B_{\alpha},\rho_{\alpha})$ by $f_{\alpha}(x)=\sum_{i=2}^{n}x_{i}^{2}$. We define $\phi=\sum_{\alpha\in\mathcal{P}'}\psi_{\alpha}f_{\alpha}$. An easy computation shows that $D\phi$ vanishes along $\tilde{\gamma}$ and in local coordinates $\Hess_{\tilde{\gamma}(t)}\phi(X,Y)=\psi_{\alpha}(x)\sum_{i=2}^{n}x_{i}y_{i}$ if $\tilde{\gamma}(t)\in B_{\alpha}$, $X=(x_{1},...,x_{n})$ and $Y=(y_{1},...,y_{n})$; and $\Hess_{\tilde{\gamma}(t)}\phi\equiv 0$ if $t\notin\bigcup_{\alpha\in\mathcal{P}'}B_{\alpha}$. In particular, if $\tilde{\gamma}(t)\in B'_{\alpha}$ for some $\alpha\in\mathcal{P}'$ then $\Hess_{\tilde{\gamma}(t)}\phi(X,X)=0$ if and only if $X\in\langle\dot{\gamma}_{E_{j}}(t)\rangle$ for some (or equivalently, for every) $j\in\{1,...,l\}$ where $\alpha=\bigcap_{i=1}^{l}\reg(\tilde{\gamma}_{E_{i}})$.

    %We will define $\phi:M\to\mathbb{R}$ in the following way. Given $1\leq i\leq k$ and an edge $\gamma_{E}$ of $\gamma_{i}$, let us choose an interior point $p_{E}\in\gamma_{E}$. Let $\eta>0$ be such that the balls $B_{E}=B(p_{E},\eta)$ are disjoint, they are the domains of certain charts $\rho_{E}:B_{E}\to\mathbb{R}^{n}$ of class $C^{q}$ (which is the regularity of $\gamma_{E}$) with $\rho_{E}(\gamma_{E}\cap B_{E})=\{(t,0,...,0):t\in\mathbb{R}\}$ and  $B_{E}\cap\gamma_{E'}=\emptyset$ if $E\neq E'$ . Let $\psi_{E}:M\to\mathbb{R}$ be a smooth function with $0\leq\psi_{E}\leq 1$, $\support(\psi_{E})\subseteq B_{E}$ and $\psi_{E}\equiv 1$ in $B(p_{E},\frac{\eta}{2})$. Let $f_{E}:B_{E}\to\mathbb{R}$ be given in local coordinates by $f_{E}(x)=\sum_{i=2}^{n}x_{i}^{2}$. We define $\phi=\sum_{E}\psi_{E}f_{E}$. An easy computation shows that $D\phi$ vanishes along $\gamma_{1},...,\gamma_{k}$ and in local coordinates $\Hess_{p}\phi(X,Y)=\psi_{E}(x)\sum_{i=2}^{n}x_{i}y_{i}$ if $p\in\gamma_{E}$, $X=(x_{1},...,x_{n})$ and $Y=(y_{1},...,y_{n})$. In particular, $\Hess_{p}\phi(X,X)=0$ if and only if $X\in\langle\dot{\gamma}_{E}(t)\rangle$.
    
    Therefore we know that $\phi$ and $D\phi$ vanish along each $\gamma_{i}$. Hence by \cite[Theorem 1.159]{Besse}, the $\gamma_{1},...,\gamma_{k}$ are still stationary with respect to $g_{\varepsilon}(x)=e^{-2\varepsilon\phi(x)}g(x)$. Fix $\gamma=\gamma_{i}:\Gamma\to M$ with set of vertices $\mathscr{V}$ and set of edges $\mathscr{E}$. We assume that $\Gamma$ is good* (i.e. every vertex has at least $3$ different incoming edges), the case when $\gamma$ is an embedded closed geodesic can be handled with the same method using the ellipticity of its Jacobi operator. As discussed in Section \ref{Section Jacobi}, the stability operator of $\gamma$ with respect to $g$ is the map $L:H^{2}(\gamma)\to L^{2}(\gamma)$ given by
    
    \begin{equation*}
        L(J)=((L_{E}(J))_{E\in\mathscr{E}},(B_{v}(J))_{v\in\mathscr{V}})
    \end{equation*}
    where
    \begin{align*}
        L_{E}(J) & =-\frac{n(E)}{l(E)}\bigg[\ddot{J}_{E}^{\perp}+R(\dot{\gamma},J_{E}^{\perp}),\dot{\gamma}\bigg]\\
        B_{v}(J) & =\sum_{(E,i):\pi_{E}(i)=v} (-1)^{i+1}\frac{n(E)}{l(E)}\dot{J}^{\perp}_{E}(i).
    \end{align*}
    
    Let us compute which is the change in the Jacobi operator along $\gamma$ when we switch from the metric $g$ to $g_{\varepsilon}$. We will denote $L^{\varepsilon}$ the operator corresponding to $g_{\varepsilon}$. Using \cite[Theorem 1.159]{Besse} and the fact that $D\phi=0$ along $\gamma_{i}$, we can see that $B_{v}^{\varepsilon}=B_{v}$ for all $v\in\mathscr{V}$, and that
    \begin{equation*}
        L_{E}^{\varepsilon}(J)=-\frac{n(E)}{l(E)}(\ddot J_{E}^{\perp}+R(\dot{\gamma},J_{E}^{\perp})\dot{\gamma}+\varepsilon\Hess\phi(J_{E}^{\perp}))
    \end{equation*}
    where the covariant derivatives and the curvature tensor $R$ are taken with respect to the metric $g$, and at each point $p\in M$, $\Hess_{p}(\phi):T_{p}M\to T_{p}M$ is the linear transformation such that the Hessian of $\phi$ at $p$ is given by $(X,Y)\mapsto \langle\Hess_{p}\phi(X),Y\rangle_{g}$.
    
     We know from Section \ref{Section Jacobi} that each $L^{\varepsilon}:H^{2}(\gamma)\to L^{2}(\gamma)$ admits a non-decreasing sequence of eigenvalues $\lambda_{1}^{\varepsilon}\leq\lambda_{2}^{\varepsilon}\leq...\leq\lambda_{Q}^{\varepsilon}\leq...$ which are characterized by
    \begin{equation*}
        \lambda_{i}^{\varepsilon}=\inf_{W}\max_{J\in W\setminus\{0\}}\frac{\langle L^{\varepsilon}(J),\iota(J)\rangle}{\langle \iota(J),\iota(J)\rangle}
    \end{equation*}
     where the infimum is taken over all $i$-dimensional subspaces $W$ of $H^{2}(\gamma)$. Also, the map $\varepsilon\mapsto L^{\varepsilon}$ is continuous; therefore $\lambda_{i}^{\varepsilon}$ varies continuously with $\varepsilon$ for every $i\in\mathbb{N}$. We will use these facts to show that for sufficiently small values of $\varepsilon>0$, $0$ is not an eigenvalue of $L^{\varepsilon}$.
    
    Let $Q$ be the unique natural number such that $0=\lambda_{Q}<\lambda_{Q+1}$ (here $\lambda_{i}:=\lambda_{i}^{0}$). Denote $S$ the sum of the eigenspaces corresponding to $\lambda_{1},...,\lambda_{Q}$. Let $J\in S$, $J\neq 0$. Then we have
    \begin{align*}
        \frac{\langle L^{\varepsilon}(J),\iota(J)\rangle}{\langle \iota(J),\iota(J)\rangle} & =\frac{\langle L(J),\iota(J)\rangle}{\langle \iota(J),\iota(J)\rangle}-\frac{\sum_{E\in\mathscr{E}}\int_{E}\frac{n(E)}{l(E)}\varepsilon \langle\Hess_{\gamma(t)}(\phi)(J_{E}^{\perp}(t)),J_{E}^{\perp}(t)\rangle_{g}dt}{\langle \iota(J),\iota(J)\rangle}\\
        & \leq -\varepsilon\sum_{E\in\mathscr{E}}\frac{n(E)}{l(E)}\frac{\int_{E}\langle\Hess_{\gamma(t)}(\phi)( J_{E}^{\perp}(t)),J_{E}^{\perp}(t)\rangle_{g}dt}{\langle \iota(J),\iota(J)\rangle}\\
        & \leq 0
    \end{align*}
    because $\Hess_{\gamma(t)}\phi\geq 0$ for every $t\in\Gamma$. 
    Suppose there is equality for some $J\in S\setminus\{0\}$. 
    Then the two inequalities should be equalities. 
    From the first one we deduce that $J$ is Jacobi along $\gamma$ for the metric $g$, and thus it verifies $\ddot J^{\perp}_{E}+R(\dot{\gamma},J^{\perp}_{E})\dot{\gamma}=0$ for every $E\in\mathscr{E}$. 
    From the second one, by considering the values of $t$ for which $\gamma(t)\in B'_{E}$, we see that $J^{\perp}_{E}$ is a null vector of $\Hess_{\gamma(t)}\phi$ along $\gamma_{E}\cap B'_{E}$ and therefore $J^{\perp}_{E}=0$ on $\gamma_{E}\cap B'_{E}$; and as it satisfies the Jacobi equation this implies $J^{\perp}_{E}=0$ for every $E\in\mathscr{E}$. Thus $J$ must be parallel and hence $J=0$ as $H^{2}(\gamma)$ does not contain non-trivial parallel vector fields. But this is a contradiction because we chose $J\in S\setminus\{0\}$. Hence we just proved that
    \begin{equation*}
        \frac{\langle L^{\varepsilon}(J),\iota(J)\rangle}{\langle \iota(J),\iota(J)\rangle}<0
    \end{equation*}
    for every $J\in S\setminus\{0\}$. As $S$ is finite dimensional and $\frac{\langle L^{\varepsilon}(J),\iota(J)\rangle}{\langle \iota(J),\iota(J)\rangle}$ is invariant under rescaling of $J$, the compactness of the unit ball in $S$ implies that there exists $c(\varepsilon)>0$ such that
    \begin{equation*}
        \frac{\langle L^{\varepsilon}(J),\iota(J)\rangle}{\langle \iota(J),\iota(J)\rangle}\leq -c(\varepsilon)
    \end{equation*}
    for every $J\in S\setminus\{0\}$. By the min-max characterization of the eigenvalues for $L^{\varepsilon}$, we see that $\lambda_{1}^{\varepsilon}\leq \lambda_{2}^{\varepsilon}\leq...\leq\lambda^{\varepsilon}_{Q}\leq c(\varepsilon)<0$. If we also choose $\varepsilon$ sufficiently small so that $\lambda_{Q+1}^{\varepsilon}>0$, we get that for $0<\varepsilon<\varepsilon(\gamma)$, $\gamma$ is nondegenerate with respect to $g_{\varepsilon}$. Taking $0<\varepsilon<\min\{\varepsilon(\gamma_{i}):1\leq i\leq k\}$ such that $g_{\varepsilon}\in\mathcal{U}$ and defining $g':=g_{\varepsilon}$ we get the desired result.

\end{proof}

\begin{lemma}\label{gradient}
Given $\eta > 0$ and $N \in \mathbb{N}$, there exists $\varepsilon > 0$ depending on $\eta$ and $N$ such that the following is true: for any Lipschitz function $f: I^N \rightarrow \mathbb{R}$ satisfying 
\begin{equation*}
    |f(x) - f(y)| \leq 2\varepsilon
\end{equation*}
for every $x,y \in I^N$, and for any subset $\mathcal{A}$ of $I^N$ of full measure, there exist $N + 1$ sequences of points $\{y_{1,m}\}_m,\cdots, \{y_{N + 1,m}\}_m$ contained in $\mathcal{A}$ and converging to a common limit $y \in (-1, 1)^N$ such that:
\begin{itemize}
    \item f is differentiable at each $y_{i, m}$,
    \item the gradients $\nabla f(y_{i,m})$ converge to $N + 1$ vectors $v_1,\cdots,v_{N + 1}$ with
    \begin{equation*}
        d_{\mathbb{R}^N}(0,\text{Conv}(v_1, \cdots, v_{N + 1})) < \eta.
    \end{equation*}
\end{itemize}
\end{lemma}

\begin{proof}
See \cite[Lemma~3]{MNS}.
\end{proof}

\section{Proof of the Main Theorem}\label{proof main thm}
Fix an $n$-dimensional closed manifold $M$. We are going to consider several choices and constructions over $M$. Let $g$ be a $C^{\infty}$ Riemannian metric on $M$. Let $\varepsilon_1 > 0$ be a positive constant such that $\varepsilon_1 <\inj(M,g)$, where $\inj(M,g)$ is the injectivity radius of $(M,g)$. Let $K$ be an integer and $\hat{B}_1,...,\hat{B}_K$ be disjoint domains in $M$, with piecewise smooth boundary, such that the union of their closures covers $M$. Let $B_1,..., B_K$ be some open neighbourhoods of $\hat{B}_1,...,\hat{B}_K$ respectively with the property that each of them is contained in a geodesic ball of radius of $\varepsilon_1$. Denote $\mathcal{M}^q$ the space of all $C^q$ Riemannian metrics on $M$. For each $1\leq k\leq K$, we define a smooth function $0 \leq \phi_k \leq 1$, $\text{spt}(\phi_k) \subseteq B_k$ such that 
\begin{equation*}
    \phi_k = 
        \begin{cases}
            \text{$1$ on $\hat{B}_k$}\\
            \text{$0$ on $B_k^c$}
        \end{cases}.
\end{equation*}
Consider also the partition of unity $\psi_k = \frac{\phi_k}{\sum_{l=1}^{K} \phi_l}$. We denote 
\begin{equation*}
    \mathcal{C}_{g, \tilde{K}, \varepsilon_1} := \{(K, \{\hat{B}_k\}, \{B_k\}, \{\phi_k\})\}
\end{equation*}
the set of all possible choices as above with $K\geq\tilde{K}$. Notice that $\mathcal{C}_{g, \tilde{K}, \varepsilon_1}$ is non-empty, as we can always find a sufficiently fine triangulation of $(M,g)$. We claim that the following property holds:

\begin{proposition}\label{PropP}
Assume that the Weyl law for $1$-cycles in $n$-manifolds holds as stated in Conjecture \ref{Weyl law 1}. Then for any metric $g\in\mathcal{M}^{\infty}$, for every $\varepsilon_1 > 0$, $\tilde{K} > 0$ and any choice of 

\begin{equation*}
    S = (K, \{\hat{B}_k\}, \{B_k\}, \{\phi_k\}) \in \mathcal{C}_{g, \tilde{K}, \varepsilon_1}
\end{equation*}
there is a metric $\tilde{g}\in\mathcal{M}^{\infty}$ arbitrarily close to $g$ in the $C^{\infty}$ topology such that the following holds: there are stationary geodesic networks $\gamma_{1},...,\gamma_{J}$ with respect to $\tilde{g}$ whose connected components are nondegenerate (according to Definition \ref{nondegcomponents}) and coefficients $\alpha_{1},...,\alpha_{J}\in[0,1]$ with $\sum_{j=1}^{J}\alpha_{j}=1$ satisfying
\begin{equation}\label{EqP}
    \Big|\sum_{j=1}^{J}\alpha_{j}\dashint_{\gamma_{j}} \psi_k dL_{\tilde{g}} - \dashint_M \psi_k d\Vol_{\tilde{g}}\Big| < \frac{\varepsilon_1}{K}
\end{equation}
for every $k=1,...,K$.
\end{proposition}

In the proof, we will need to measure the distance between two rescaled functions. In order to do that, we introduce the following definition.
\begin{definition}\label{epsilon close}
We say that two functions $f,g:(-\delta,\delta)^{K}\to\mathbb{R}$ are $\varepsilon$-close if
\begin{equation*}
    \Vert \frac{1}{\delta}f_{\delta}-\frac{1}{\delta}g_{\delta} \Vert_{\infty}<\varepsilon
\end{equation*}
where $f_{\delta},g_{\delta}:(-1,1)^{K}\to\mathbb{R}$ are given by $f_{\delta}(s)=f(\delta s)$ and $g_{\delta}(s)=g(\delta s)$.
\end{definition}

\begin{remark}\label{rkfdelta}
Observe that $\frac{1}{\delta}f_{\delta}$ is differentiable at $s\in(-1,1)^{K}$ if and only if $f$ is differentiable at $\delta s\in(-\delta,\delta)^{K}$ and in that case $\nabla(\frac{1}{\delta}f_{\delta})(s)= \nabla f(\delta s)$.
\end{remark}

\begin{proof}[Proof of Proposition \ref{PropP}]
Let $g\in\mathcal{M}^{\infty}$, $\tilde{K}\in\mathbb{N}$ and $\varepsilon_{1}>0$. Fix $(K,\{\hat{B}_{k}\},\{B_{k}\},\{\phi_{k}\})\in\mathcal{C}_{g,\tilde{K},\varepsilon_{1}}$. Let $\mathcal{U}$ be a $C^{\infty}$ neighborhood of $g$. Choose $\varepsilon'_{0} > 0$ sufficiently small and $q \geq K + 3$ sufficiently large so that if $g' \in \mathcal{M}^{\infty}$ satisfies $\Vert g - g'\Vert_{C^q} < \varepsilon'_{0}$, then $g' \in \mathcal{U}$. Let $\varepsilon'\leq\varepsilon'_{0}$ be a positive real number (which we will have to shrink later in the argument). Our goal is to show that there exists $\tilde{g}\in\mathcal{M}^{\infty}$ such that $\Vert \tilde{g}-g\Vert_{C^{q}}<\varepsilon'_{0}$ and (\ref{EqP}) holds for some stationary geodesic nets $\gamma_{1},...,\gamma_{J}$ (whose connected components are nondegenerate)  with respect to $\tilde{g}$ and some coefficients $\alpha_{1},...,\alpha_{J}$.

Consider the following $K$-parameter family of metrics. For a $t = (t_{1},...,t_{K}) \in (-1,1)^K$, we define
\begin{equation*}
    \hat{g}(t) = e^{2\sum_{k}t_{k}\psi_{k}}g.
\end{equation*}
At $t = 0$, for each $k$, we have
\begin{align*}
    \frac{\partial}{\partial t_{k}}\big|_{t=0}\Vol(M,\hat{g}(t))& =\frac{\partial}{\partial t_{k}}\big|_{t=0}\int_{M}(e^{2\sum_{k}t_{k}\psi_{k}(x)})^{\frac{n}{2}}\dVol_{g}\\
    & =\int_{M}n\psi_{k}(x)\dVol_{g}.
\end{align*}
As $t$ goes to zero, we have the following expansion
\begin{equation}\label{TayVol}
    \Vol(M,\hat{g}(t))^{\frac{1}{n}}=\Vol(M,g)^{\frac{1}{n}}+\sum_{k=1}^{K}t_{k}\Vol(M,g)^{-\frac{n-1}{n}}\int_{M}\psi_{k}(x)\dVol_{g}+R(t) 
\end{equation}
where $|R(t)|\leq C_{1}\Vert t\Vert^{2}$ if $t\in(-1,1)^{K}$, where $C_{1}$ is a constant which depends only on $g$ (this can be checked by computing the second order partial derivatives of $t\mapsto \Vol(M,\hat{g}(t))^{\frac{1}{n}}$ and using the fact that $e^{-n}\Vol(M,g)\leq\Vol(M,\hat{g}(t))\leq e^{n}\Vol(M,g)$ as $\Vol(M,\hat{g}(t))=\int_{M}e^{n\sum_{k}t_{k}\psi_{k}(x)}\dVol_{g}$). Following \cite{MNS} we can define the following function
\begin{equation*}
    f_{0}(t)=\frac{\Vol(M,\hat{g}(t))^{\frac{1}{n}}}{\Vol(M,g)^{\frac{1}{n}}}-\sum_{k=1}^{K}t_{k}\dashint_{M}\psi_{k}(x)\dVol_{g}.
\end{equation*}

Because of (\ref{TayVol}), $|f_{0}(t)-1|=|\frac{R(t)}{\Vol(M,g)^{\frac{1}{n}}}|\leq C_{2}\Vert t\Vert^{2}$ for every $t\in(-1,1)^{K}$; where $C_{2}=\frac{C_{1}}{\Vol(M,g)^{\frac{1}{n}}}$ depends only on $g$ (as $C_{1}$ and other constants $C_{i}$ to be defined later). 

By the previous, $f_{0}$ is $C_{2}\varepsilon'$-close to $1$ in $(-\delta,\delta)^{K}$ if $\delta<\varepsilon'$ (see Definition \ref{epsilon close}). Let $\delta<\varepsilon'$ be such that $\hat{g}:(-\delta,\delta)^{K}\to\mathcal{M}^{q}$ is an embedding and $\Vert\hat{g}(t)-g\Vert_{C^{q}}<\frac{\varepsilon'}{2}$ for every $t\in(-\delta,\delta)^{K}$. We can slightly perturb $\hat{g}$ in the $C^{\infty}$ topology to another embedding $g':(-\delta,\delta)^{K}\to\mathcal{M}^{q}$ applying Proposition \ref{Perturbation}. We can assume $\Vert g'(t)-\hat{g}(t)\Vert_{C^{q}}<\frac{\varepsilon'}{2}$ and $\Vert\frac{\partial g'}{\partial v}-\frac{\partial \hat{g}}{\partial v}\Vert_{C^{q}}<\varepsilon'$ for every $t\in(-\delta,\delta)^{K}$ and $v\in\mathbb{R}^{K}:|v|=1$. Consider the function
\begin{equation*}
    f_{1}(t)=\frac{\Vol(M,g'(t))^{\frac{1}{n}}}{\Vol(M,g)^{\frac{1}{n}}}-\sum_{k}t_{k}\dashint_{M}\psi_{k}(x)\dVol_{g}.
\end{equation*}
By the properties of $g'$, there exists $C_{3}>0$ such that $f_{1}$ is $C_{3}\varepsilon'$-close to the constant function equal to $1$ on $(-\delta, \delta)^{K}$.

Now we will use the assumption that the Weyl law for $1$-cycles in $n$-manifolds holds, which means that
\begin{equation}\label{Weyl's law eq}
    \lim_{p \rightarrow \infty}\omega^1_p(M^n, g)p^{-\frac{n - 1}{n}} = \alpha(n, 1)\Vol(M^n, g)^{\frac{1}{n}}.
\end{equation}

The normalized $p$-widths $p^{-\frac{n-1}{n}}\omega_p^{1}(g'(t))$ are uniformly Lipschitz continuous on $(-\delta,\delta)^{K}$ by \cite[Lemma 3.4]{LS}. Hence, by (\ref{Weyl's law eq}) the sequence of functions \linebreak $t \mapsto p^{-\frac{n-1}{n}}\omega^{1}_{p}(M,g'(t))$ converges uniformly to the function $t \mapsto a(n)\Vol(M,g'(t))^{\frac{1}{n}}$. This implies that for the previously defined $\delta>0$, there exists $p_{0}\in\mathbb{N}$ such that $p\geq p_{0}$ implies
\begin{equation*}
    |p^{-\frac{n-1}{n}}\omega_{p}^{1}(M,g'(t))-\alpha(n,1)\Vol(M,g'(t))^{\frac{1}{n}}|<\delta\varepsilon'
\end{equation*}
and hence
\begin{equation*}
    |\frac{\omega^{1}_{p}(M,g'(t))}{\alpha(n,1)p^{\frac{n-1}{n}}\Vol(M,g)^{\frac{1}{n}}}-\frac{\Vol(M,g'(t))^{\frac{1}{n}}}{\Vol(M,g)^{\frac{1}{n}}}|<C_{4}\delta\varepsilon'
\end{equation*}
for every $t\in(-\delta,\delta)^{K}$. The previous means that $h(t)=\frac{\omega^{1}_{p}(M,g'(t))}{\alpha(n,1)p^{\frac{n-1}{n}}\Vol(M,g)^{\frac{1}{n}}}-\frac{\Vol(M,g'(t))^{\frac{1}{n}}}{\Vol(M,g)^{\frac{1}{n}}}$ is $C_{4}\varepsilon'$-close to $0$ in $(-\delta,\delta)^{K}$ and therefore as $f_{1}$ is $C_{3}\varepsilon'$-close to $1$, by triangle inequality we have that
\begin{equation*}
    f_{2}(t)=\frac{\omega^{1}_{p}(M,g'(t))}{\alpha(n,1)p^{\frac{n-1}{n}}\Vol(M,g)^{\frac{1}{n}}}-\sum_{k=1}^{K}t_{k}\dashint_{M}\psi_{k}(x)\dVol_{g}
\end{equation*}
is $C_{5}\varepsilon'$-close to $1$ if $p\geq p_{0}$, for some $C_{5}>0$.

On the other hand, by our choice of $g'$ using Proposition \ref{Perturbation}, there exists a full measure subset $\mathcal{A}\subseteq(-\delta,\delta)^{K}$ such that for each $t\in\mathcal{A}$ and $p\in\mathbb{N}$ the map $t\mapsto\omega^{1}_{p}(g'(t))$ is differentiable at $t$ and there exists a stationary geodesic net $\gamma_{p}(t)$ with respect to $g'(t)$ so that
\begin{enumerate}
    \item $\omega_{p}^{1}(g'(t))=\length_{g'(t)}(\gamma_{p}(t))$
    \item $\frac{\partial}{\partial v}(\omega_{p}^{1}\circ g'(s))|_{s=t}=\frac{1}{2}\int_{\gamma_{p}(t)}\trace_{\gamma_{p}(t),g'(t)}\frac{\partial g'}{\partial v}(t)\dL_{g'(t)}$.
\end{enumerate}
Define $f_{3}:(-1,1)^{K}\to\mathbb{R}$ as $f_{3}(t)=\frac{1}{\delta}f_{2}(\delta t)$. We know that $\Vert f_{3}-1\Vert_{\infty,(-1,1)^{K}}<C_{5}\varepsilon'$. Now we want to use Lemma \ref{gradient}. In order to do that we will need to impose more restrictions on $\varepsilon'$. Let $\eta>0$. Let $\varepsilon>0$ be the one depending on $\eta$ and $N=K$ according to Lemma \ref{gradient}. Choose $\varepsilon'$ small enough so that $C_{5}\varepsilon'<\varepsilon$, $\varepsilon'<\eta$ and $\varepsilon'\leq\varepsilon'_{0}$. Observe that this allows us to define $\delta>0$ and $p_{0}\in\mathbb{N}$ with all the properties in the construction above. Then we have
\begin{equation*}
    |f_{3}(x)-f_{3}(y)|\leq 2\varepsilon
\end{equation*}
for every $x,y\in(-1,1)^{K}$. As $f_{3}$ is Lipschitz, we can apply Lemma \ref{gradient} to $f_{3}$ and the full measure subset $\mathcal{A'}=\{\frac{t}{\delta}:t\in\mathcal{A}\}$. After passing to $(-\delta,\delta)^{K}$ by rescaling and using Remark \ref{rkfdelta}, we get $K+1$ sequences of points $\{s_{1,m}\}_{m},...,\{s_{K+1,m}\}_{m\in\mathbb{N}}$ contained in $\mathcal{A}$ and converging to a common limit $s\in(-\delta,\delta)^{K}$ such that:
\begin{enumerate}
    \item $f_{2}$ is differentiable at each $s_{j,m}$.
    \item The gradients $\nabla f_{2}(s_{j,m})$ converge to $K+1$ vectors $v_{1},...,v_{K+1}$ with
    \begin{equation*}
        d_{\mathbb{R}^{N}}(0,\text{Conv}(v_{1},...,v_{K+1}))<\eta.
    \end{equation*}
\end{enumerate}

Let $\alpha_{1},...,\alpha_{K+1}\in[0,1]$ be such that $\sum_{j=1}^{K+1}\alpha_{j}=1$ and $|\sum_{j=1}^{K+1}\alpha_{j}v_{j}|<\eta$. Then if $m$ is sufficiently large,
\begin{equation*}
    |\sum_{j=1}^{K+1}\alpha_{j}\nabla f_{2}(s_{j,m})|<\eta
\end{equation*}
and hence
\begin{equation*}
    |\sum_{j=1}^{K+1}\alpha_{j}\frac{\partial f_{2}}{\partial t_{k}}(s_{j,m})|<\eta
\end{equation*}
for every $k=1,...,K$. But using the definition of $f_{2}$ and denoting $\gamma_{j,m}=\gamma_{p}(s_{j,m})$,

\begin{align*}
    \frac{\partial f_{2}}{\partial t_{k}}(s_{j,m}) & =\frac{\frac{\partial}{\partial t_{k}}\omega_{p}^{1}(M,g'(s))|_{s=s_{j,m}}}{\alpha(n,1)\Vol(M,g)^{\frac{1}{n}}p^{\frac{n-1}{n}}}-\dashint_{M}\psi_{k}(x)\dVol_{g}\\
    & = \frac{\int_{\gamma_{j,m}}\trace_{\gamma_{j,m},g'(s_{j,m})}\frac{\partial g'}{\partial t_{k}}(s_{j,m})\dL_{g'(s_{j,m})}}{2\alpha(n,1)\Vol(M,g)^{\frac{1}{n}}p^{\frac{n-1}{n}}}-\dashint_{M}\psi_{k}(x)\dVol_{g}.
\end{align*}

As the lengths $\length_{g'(s_{j,m})}(\gamma_{j,m})=\omega_{p}(g'(s_{j,m}))$ of the $\gamma_{j,m}$'s are uniformly bounded, passing to a subsequence we can obtain stationary geodesic networks $\gamma_{1},...,\gamma_{K+1}$ with respect to $g'(s)$ verifying
\begin{equation}\label{limit nets}
    \lim_{m\to\infty}\gamma_{j,m}=\gamma_{j}
\end{equation}
in the varifold topology for every $j=1,...,K+1$. Hence from the previous,

\begin{equation*}
    |\sum_{j=1}^{K+1}\alpha_{j}\frac{\int_{\gamma_{j}}\trace_{\gamma_{j},g'(s)}\frac{\partial g'}{\partial t_{k}}(s)\dL_{g'(s)}}{2\alpha(n,1)\Vol(M,g)^{\frac{1}{n}}p^{\frac{n-1}{n}}}-\dashint_{M}\psi_{k}(x)\dVol_{g}|\leq \eta
\end{equation*}
for every $k=1,...,K$. Using that $\Vert\hat{g}(t)-g\Vert_{C^{q}}<\frac{\varepsilon'}{2}$, $\Vert g'(t)-\hat{g}(t)\Vert_{C^{q}}<\frac{\varepsilon'}{2}$ and $\Vert\frac{\partial g'}{\partial v}-\frac{\partial \hat{g}}{\partial v}\Vert_{C^{q}}<\varepsilon'$ for every $t\in(-\delta,\delta)^{K}$ and $v\in\mathbb{R}^{K}:|v|=1$; and the fact that $\varepsilon'<\eta$, we can see that there exists a constant $C_{6}>0$ such that

\begin{equation*}
    |\sum_{j=1}^{K+1}\alpha_{j}\frac{\int_{\gamma_{j}}{\trace_{\gamma_{j},\hat{g}(s)}\frac{\partial \hat{g}}{\partial t_{k}}(s)}\dL_{\hat{g}(s)}}{2\alpha(n,1)\Vol(M,g'(s))^{\frac{1}{n}}p^{\frac{n-1}{n}}}-\dashint_{M}\psi_{k}(x)\dVol_{g'(s)}|\leq C_{6}\eta.
\end{equation*}
By definition of $\hat{g}$, $\frac{\partial\hat{g}}{\partial t_{k}}(s)=2\psi_{k}\hat{g}(s)$ thus
\begin{align}\label{Eq1}
    |\sum_{j=1}^{K+1}\alpha_{j}\frac{\int_{\gamma_{j}}\psi_{k}\dL_{\hat{g}(s)}}{\alpha(n,1)\Vol(M,g'(s))^{\frac{1}{n}}p^{\frac{n-1}{n}}}-\dashint_{M}\psi_{k}(x)\dVol_{g'(s)}|\leq C_{6}\eta.
\end{align}
Combining (\ref{Eq1}) with the fact that $\Vert g'(s)-\hat{g}(s)\Vert_{C^{q}}<\frac{\varepsilon'}{2}$,
\begin{equation}\label{Eq2}
    |\sum_{j=1}^{K+1}\alpha_{j}\frac{\int_{\gamma_{j}}\psi_{k}\dL_{g'(s)}}{\alpha(n,1)\Vol(M,g'(s))^{\frac{1}{n}}p^{\frac{n-1}{n}}}-\dashint_{M}\psi_{k}(x)\dVol_{g'(s)}|\leq C_{7}\eta.
\end{equation}
But we know that $\L_{g'(s)}(\gamma_{j})=\omega_{p}^{1}(g'(s))$ for every $j=1,...,K+1$, so
\begin{align*}
    |\frac{\int_{\gamma_{j}}\psi_{k}\dL_{g'(s)}}{\alpha(n,1)\Vol(M,g'(s))^{\frac{1}{n}}p^{\frac{n-1}{n}}}-\dashint_{\gamma_{j}}\psi_{k}\dL_{g'(s)}| & = \\
    |\dashint_{\gamma_{j}}\psi_{k}\dL_{g'(s)}||\frac{\omega_{p}^{1}(g'(s))}{\alpha(n,1)\Vol(M,g'(s))^{\frac{1}{n}}p^{\frac{n-1}{n}}}-1| &\leq\\
    |\frac{\omega_{p}^{1}(g'(s))}{\alpha(n,1)\Vol(M,g'(s))^{\frac{1}{n}}p^{\frac{n-1}{n}}}-1| & \leq \eta
\end{align*}
if $p\geq p_{1}$ for some $p_{1}\in\mathbb{N}$, because of the Weyl law and the fact $0\leq\psi_{k}\leq 1$. Hence from (\ref{Eq2}),
\begin{equation*}
    |\sum_{j=1}^{K+1}\alpha_{j}\dashint_{\gamma_{j}}\psi_{k}\dL_{g'(s)}-\dashint_{M}\psi_{k}\dVol_{g'(s)}|\leq C_{8}\eta
\end{equation*}
for some constant $C_{8}$ depending only on $g$. Let us take $\eta=\frac{\varepsilon_{1}}{2C_{8}K}$ and $p\geq\max\{p_{0},p_{1}\}$. 

%Notice that $\Vert g'(s)-g\Vert_{C^{q}}\leq\Vert g'(s)-\hat{g}(s)\Vert_{C^{q}}+\Vert\hat{g}(s)-g\Vert_{q}<\frac{\varepsilon'}{2}+\frac{\varepsilon'}{2}<\varepsilon_{0}'$. Thus the metric $g'(s)$ verifies the required assumptions except perhaps the non-degeneracy condition and the $C^{\infty}$ regularity. Regarding the first one, we can do a small perturbation of $g'(s)$ to get another $C^{q}$ metric $\overline{g}$ with respect to which the $\gamma_{i}$s are nondegenerate

Notice that $\Vert g'(s)-g\Vert_{C^{q}}\leq\Vert g'(s)-\hat{g}(s)\Vert_{C^{q}}+\Vert\hat{g}(s)-g\Vert_{C^{q}}<\frac{\varepsilon'}{2}+\frac{\varepsilon'}{2}<\varepsilon_{0}'$. Let us represent each $\gamma_{i}$ as a map $\gamma_{i}:\Gamma_{i}\to M$ where each connected component of the weighted multigraph $\Gamma_{i}$ is good and the restrictions of $\gamma_{i}$ to those connected components are embedded (here we are using Remark \ref{unionofregular}). The metric $g'(s)$ has all the properties required by Proposition \ref{PropP} except that the components of the $\gamma_{i}$'s may not be non-degenerate and may not be $C^{\infty}$ (in principle they are only $C^{q}$). Using Proposition \ref{Perturbation nondeg}, we can change $g'(s)$ for another $C^{q}$ metric $\overline{g}$ which still verifies $\Vert\overline{g}-g\Vert_{C^{q}}<\varepsilon_{0}'$, and has the property that $\gamma_{1},...,\gamma_{K+1}$ are non-degenerate stationary geodesic nets with respect to $\overline{g}$. If on top of that we apply the Implicit Function Theorem (see \cite[Lemma~2.6]{LS}), we can find a $C^{\infty}$ metric $\tilde{g}$ close enough to $\overline{g}$ in the $C^{q}$ topology so that $\Vert\tilde{g}-g\Vert_{C^{q}}<\varepsilon_{0}$ which admits stationary geodesic networks $\tilde{\gamma}_{1},...,\tilde{\gamma}_{k+1}$ whose connected components are nondegenerate and verify
\begin{equation*}
    |\sum_{j=1}^{K+1}\alpha_{j}\dashint_{\tilde{\gamma}_{j}}\psi_{k}\dL_{\tilde{g}}-\dashint_{M}\psi_{k}\dVol_{\tilde{g}}|< \frac{\varepsilon_{1}}{K}
\end{equation*}
for every $k=1,...,K+1$. This completes the proof.

%Also, from the proof of Proposition \ref{Perturbation} the metric $g'(t)$ is bumpy for every $t\in(-\delta,\delta)^{K}$, and therefore we can represent $\gamma_{i}$ as a map $\gamma_{i}:\Gamma_{i}\to M$ where each connected component of the weighted multigraph $\Gamma_{i}$ is good and the restrictions of $\gamma_{i}$ to those connected components are embedded and therefore nondegenerate (here we are using Remark \ref{unionofregular}). Therefore, the metric $g'(s)$ has all the properties required by Proposition \ref{PropP} except that it may not be $C^{\infty}$ (in principle it is only $C^{q}$). But applying the Implicit Function Theorem (see \cite[Lemma~2.6]{LS}), we can find a $C^{\infty}$ metric $\tilde{g}$ close enough to $g'(s)$ in the $C^{q}$ topology so that $\Vert\tilde{g}-g\Vert_{C^{q}}<\varepsilon_{0}$ and it also admits stationary geodesic networks $\tilde{\gamma}_{1},...,\tilde{\gamma}_{k+1}$ whose connected components are nondegenerate and verifying 
%\begin{equation*}
    %|\sum_{j=1}^{K+1}\alpha_{j}\dashint_{\tilde{\gamma}_{j}}\psi_{k}\dL_{\tilde{g}}-\dashint_{M}\psi_{k}\dVol_{\tilde{g}}|< \frac{\varepsilon_{1}}{K}
%\end{equation*}
%for every $k=1,...,K+1$, which completes the proof.

\end{proof}

Now we will show that Proposition \ref{PropP} implies Theorem \ref{thm3}. Given $g\in\mathcal{M}^{\infty}$, $\varepsilon_{1}>0$, $\tilde{K}\in\mathbb{N}$ and $S\in\mathcal{C}_{g,\tilde{K},\varepsilon_{1}}$ we will denote $\mathcal{M}(g,\varepsilon_{1},\tilde{K},S)$ the set of all metrics $\tilde{g}\in\mathcal{M}^{\infty}$ such that $\Vert\tilde{g}-g\Vert_{C^{q}}<\varepsilon_{1}$ (computed with respect to $g$) and there exist stationary geodesic networks $\gamma_{1},...,\gamma_{J}$ with respect to $\tilde{g}$ whose connected components are nondegenerate (according to Definition \ref{nondegcomponents}) and coefficients $\alpha_{1},...,\alpha_{J}\in[0,1]$ with $\sum_{j=1}^{J}\alpha_{j}=1$ such that (\ref{EqP}) holds for every $k=1,...,K$. By the Implicit Function Theorem, $\mathcal{M}(g,\varepsilon_{1},\tilde{K},S)$ is open (see \cite[Lemma~2.6]{LS}). Therefore given $\varepsilon_{1}>0$ and $\tilde{K}\in\mathbb{N}$ the set
\begin{equation*}
    \mathcal{M}(\varepsilon_{1},\tilde{K})=\bigcup_{g\in\mathcal{M}^{\infty}}\bigcup_{S\in\mathcal{C}_{g,\varepsilon_{1},\tilde{K}}}\mathcal{M}(g,\varepsilon_{1},K,S)
\end{equation*}
is open and by Proposition \ref{PropP} it is also dense in $\mathcal{M}^{\infty}$. Define
\begin{equation*}
    \tilde{\mathcal{M}}=\bigcap_{m\in\mathbb{N}}\mathcal{M}(\frac{1}{m},m)
\end{equation*}
which is a generic subset of $\mathcal{M}^{\infty}$ in the Baire sense. We are going to prove that if $\tilde{g}\in\tilde{\mathcal{M}}$ then there exists a sequence of equidistributed stationary geodesic networks with respect to $\tilde{g}$.

Fix $\tilde{g}\in\tilde{\mathcal{M}}$. By definition, given $m\in\mathbb{N}$ there exists $g\in\mathcal{M}^{\infty}$ such that $\tilde{g}\in\mathcal{M}(g,\frac{1}{m},m,S)$ for some $S\in\mathcal{C}_{g,\frac{1}{m},m}$. Therefore, $\tilde{g}$ belongs to a $\frac{1}{m}$ neighborhood of $g$ in the $C^{K}$ topology; and there exist $J=J_{m}\in\mathbb{N}$, stationary geodesic networks $\gamma_{m,1},...,\gamma_{m,J_{m}}$ with respect to $\tilde{g}$ and coefficients $\alpha_{m,1},...,\alpha_{m,J_{m}}\in[0,1]$ with $\sum_{j=1}^{J_{m}}\alpha_{m,j}=1$ satisfying
\begin{equation}\label{Eq3}
    |\sum_{j=1}^{J_{m}}\alpha_{m,j}\dashint_{\gamma_{m,j}}\psi_{k}(x)\dL_{\tilde{g}}-\dashint_{M}\psi_{k}(x)\dVol_{\tilde{g}}|<\frac{1}{mK}
\end{equation}
for every $k=1,...,K$. Let $f\in C^{\infty}(M,\mathbb{R})$. We want to obtain a formula analogous to the previous one but replacing $\psi_{k}$ by $f$, which will imply the following proposition.

\begin{proposition}\label{Propmj}
Let $\tilde{g}\in\tilde{\mathcal{M}}$. For each $m\in\mathbb{N}$, there exists $J=J_{m}$ depending on $m$, integers $\{c_{m,j}\}_{1\leq j\leq J_{m}}$ and stationary geodesic networks $\{\gamma_{m,j}\}_{1\leq j\leq J_{m}}$ such that
\begin{equation*}
    |\frac{\sum_{j = 1}^{J_{m}}c_{m,j}\int_{\gamma_{m,j}}f\dL_{\tilde{g}}}{\sum_{j = 1}^{J_m}c_{m,j}\length_{\tilde{g}}(\gamma_{m,j})}- \dashint_{M}f\dVol_{\tilde{g}}|\leq\frac{D(f)}{m}
\end{equation*}
for every $f\in C^{\infty}(M,\mathbb{R})$, where $D(f)>0$ is a constant depending only on $f$ and the metric $\tilde{g}$.
\end{proposition}

\begin{proof}
Given $m\in\mathbb{N}$, consider as above $g\in\mathcal{M}^{\infty}$ and $S\in\mathcal{C}_{g,\frac{1}{m},m}$ such that $\tilde{g}\in\mathcal{M}(g,\frac{1}{m},m,S)$. Define $J_{m}\in\mathbb{N}$, stationary geodesic networks $\gamma_{m,1},...,\gamma_{m,J_{m}}$ with respect to $\tilde{g}$ and coefficients $\alpha_{m,1},...,\alpha_{m,J_{m}}$ such that (\ref{Eq3})  holds. Taking $S=(K,\{\hat{B}_{k}\}_{k},\{B_{k}\}_{k},\{\phi_{k}\}_{k})\in\mathcal{C}_{g,\frac{1}{m},m}$ into account, let us choose points $q_{1},...,q_{K}$ with $q_{k}\in\hat{B}_{k}$ for each $k=1,...,K$. The idea will be to approximate the integral of $f(x)$ by the integral of the function $\sum_{k=1}^{K}f(q_{k})\psi_{k}(x)$. First of all, by using (\ref{Eq3}) we can see that
\begin{equation}\label{Eq4}
    |\sum_{j=1}^{J_{m}}\alpha_{m,j}\dashint_{\gamma_{m,j}}[\sum_{k=1}^{K}f(q_{k})\psi_{k}(x)]d\length_{\tilde{g}}-\dashint_{M}[\sum_{k=1}^{K}f(q_{k})\psi_{k}(x)]d\Vol_{\tilde{g}}|<\frac{D_{1}}{m}
\end{equation}
where $D_{1}=\Vert f\Vert_{\infty}=\max\{f(x):x\in M\}$ depends only on $f$ (and not on $m$, $g$ or $S$). On the other hand, given $x\in M$
\begin{align*}
    |f(x)-\sum_{k=1}^{K}f(q_{k})\psi_{k}(x)| & =|f(x)\sum_{k=1}^{K}\psi_{k}(x)-\sum_{k=1}^{K}f(q_{k})\psi_{k}(x)|\\
    & = |\sum_{k=1}^{K}f(x)\psi_{k}(x)-f(q_{k})\psi_{k}(x)|\\
    & \leq \sum_{k:x\in B_{k}}|f(x)-f(q_{k})||\psi_{k}(x)|\\
    & = \sum_{k:x\in B_{k}}|\nabla_{\tilde{g}} f(c_{k})|d_{\tilde{g}}(x,q_{k})\psi_{k}(x)\\
    &\leq \frac{2\Vert\nabla_{\tilde{g}} f\Vert_{\infty}}{m}\sum_{k=1}^{K}\psi_{k}(x)\\
    & =\frac{2\Vert\nabla_{\tilde{g}} f\Vert_{\infty}}{m}.
\end{align*}
We used the Mean Value Theorem and the fact that $\text{supp}(\psi_{k})\subseteq B_{k}$ and $\text{diam}_{\tilde{g}}(B_{k})\leq 2\text{diam}_{g}(B_{k})\leq\frac{2}{m}$ for every $i$. Combining this and (\ref{Eq4}) we get

\begin{equation}\label{Eq5}
    |\sum_{j=1}^{J_{m}}\alpha_{m,j}\dashint_{\gamma_{m,j}}f\dL_{\tilde{g}}-\dashint_{M}f\dVol_{\tilde{g}}|<\frac{D_{2}}{m}
\end{equation}
where $D_{2}$ depends only on $f$ and $\tilde{g}$. Let us choose integers $c_{m,j},d_{m}\in\mathbb{N}$ such that
\begin{equation*}
    |\frac{\alpha_{m,j}}{\length_{\tilde{g}}(\gamma_{m,j})}-\frac{c_{m,j}}{d_{m}}|<\frac{1}{mJ_{m}\length_{\tilde{g}}(\gamma_{m.j})}.
\end{equation*}
Then it holds
\begin{align*}
    |\sum_{j=1}^{J_{m}}\alpha_{m,j}\dashint_{\gamma_{m,j}}f\dL_{\tilde{g}}-\sum_{j=1}^{J_{m}}\frac{c_{m,j}}{d_{m}}\int_{\gamma_{m,j}}f\dL_{\tilde{g}}|& \leq\sum_{j=1}^{J_{m}}|\frac{\alpha_{m,j}}{\length_{\tilde{g}(\gamma_{m.j})}}-\frac{c_{m_{j}}}{d_{m}}||\int_{\gamma_{m,j}}f\dL_{\tilde{g}}|\\
    &\leq \sum_{j=1}^{J_{m}}\frac{1}{mJ_{m}\length_{\tilde{g}}(\gamma_{m,j})}\Vert f\Vert_{\infty}\length_{\tilde{g}}(\gamma_{m,j})\\
    &= \frac{D_{1}}{m}
\end{align*}
and hence by (\ref{Eq5}) and triangle inequality we get
\begin{equation*}
    |\sum_{j=1}^{K_{m}}\frac{c_{m,j}}{d_{m}}\int_{\gamma_{m,j}}f\dL_{\tilde{g}}-\dashint_{M}f\dVol_{\tilde{g}}|<\frac{D_{3}}{m}
\end{equation*}
where $D_{3}=D_{2}+D_{1}$ depends only on $f$ and $\tilde{g}$. On the other hand,

\begin{align*}
    |\sum_{j=1}^{J_{m}}\frac{c_{m,j}}{d_{m}}\int_{\gamma_{m,j}}f\dL_{\tilde{g}}-\frac{\sum_{j=1}^{J_{m}}c_{m,j}\int_{\gamma_{m,j}}f\dL_{\tilde{g}}}{\sum_{j=1}^{J_{m}}c_{m,j}\length_{\tilde{g}}(\gamma_{m,j})}| &  \leq\\
    |\frac{1}{d_{m}}-\frac{1}{\sum_{j=1}^{J_{m}}c_{m,j}\length_{\tilde{g}}(\gamma_{m,j})}||\sum_{j=1}^{J_{m}}c_{m,j}\int_{\gamma_{m,j}}f\dL_{\tilde{g}}| & \leq\\
    |\frac{1}{d_{m}}-\frac{1}{\sum_{j=1}^{J_{m}}c_{m,j}\length_{\tilde{g}}(\gamma_{m,j})}|\sum_{j=1}^{J_{m}}c_{m,j}\Vert f\Vert_{\infty}\length_{\tilde{g}}(\gamma_{m,j})& =\\
    D_{1}|\sum_{j=1}^{J_{m}}\frac{c_{m,j}}{d_{m}}\length_{\tilde{g}}(\gamma_{m,j})-1| &\leq\frac{D_{1}}{m}
\end{align*}
because $|\sum_{j=1}^{J_{m}}\frac{c_{m,j}}{d_{m}}\length_{\tilde{g}}(\gamma_{m,j})-1|<\frac{1}{m}$. Hence
\begin{equation*}
    |\frac{\sum_{j = 1}^{J_m}c_{m,j}\int_{\gamma_{m,j}}f\dL_{\tilde{g}}}{\sum_{j = 1}^{J_m}c_{m,j}\length_{\tilde{g}}(\gamma_{m,j})}- \dashint_{M}f\dVol_{\tilde{g}}|\leq\frac{D_{4}}{m}
\end{equation*}
for a constant $D_{4}$ depending only on $f$ and $\tilde{g}$, as desired.
\end{proof}

Given $\tilde{g}\in\tilde{\mathcal{M}}$, using Proposition \ref{Propmj} we can find a sequence of finite lists of connected embedded stationary geodesic nets $\{\beta_{m, 1},..., \beta_{m, K_{m}}\}_{m \in \mathbb{N}}$ with respect to $\tilde{g}$ satisfying the following: given $f\in C^{\infty}(M,\mathbb{R})$, if we denote $X_{m,j} = \int_{\beta_{m,j}}f\dL_{\tilde{g}}$ and $\Bar{X}_{m,j} = \length_{\tilde{g}}(\beta_{m,j})$, then
\begin{equation}
    |\frac{\sum_{j = 1}^{K_{m}}X_{m,j}}{\sum_{j = 1}^{K_{m}}\Bar{X}_{m,j}} - \alpha| \leq\frac{D(f)}{m}
\end{equation}
where $\alpha = \dashint_{M}f\dVol_{\tilde{g}}$ and $D(f)$ is a constant depending only on $f$. The lists $\{\beta_{m,j}\}_{1\leq j\leq K_{m}}$ are obtained from the lists $\{\gamma_{m,j}\}_{1\leq j\leq J_{m}}$ and the coefficients $\{c_{m,j}\}_{1\leq j\leq J_{m}}$ from Proposition \ref{Propmj} by decomposing each $\gamma_{m,j}$ as a union of embedded stationary geodesic networks whose domain is a good weighted multigraph (see Remark \ref{unionofregular}) and listing each of them $c_{m,j}$ times. From the $X_{m,j}$'s and the $\Bar{X}_{m,j}$'s, we want to construct two sequences $\{Y_i\}_{i \in \mathbb{N}}, \{\Bar{Y}_i\}_{i \in \mathbb{N}}$ such that
\begin{itemize}
    \item For all $i$, there exist integers $m(i), j(i)$ (chosen independently of $f$) with $Y_{i} = X_{m(i),j(i)}$ and $\Bar{Y}_i = \Bar{X}_{m(i),j(i)}$,
    \item It holds
    \begin{equation*}
        \lim_{k \rightarrow \infty}\frac{\sum_{i = 1}^k Y_i}{\sum_{i = 1}^k \Bar{Y}_i} = \alpha.
    \end{equation*}
\end{itemize}

This can be done as in \cite[p.~437-439]{MNS} and gives us a sequence $\{\gamma_{i}\}_{i\in\mathbb{N}}$ of connected embedded stationary geodesic networks with respect to $\tilde{g}$ (defined as $\gamma_{i}=\beta_{m(i),j(i)}$), which is constructed independently of the constant $D(f)$. It holds

\begin{equation*}
    \lim_{k\to\infty}\frac{\sum_{i=1}^{k}\int_{\gamma_{i}}f\dL_{\tilde{g}}}{\sum_{i=1}^{k}\length_{\tilde{g}}(\gamma_{i})}=\dashint_{M}f\dVol_{\tilde{g}}
\end{equation*}
for every $f\in C^{\infty}(M,\mathbb{R})$. This gives us the desired equidistribution result and completes the proof of Theorem \ref{thm3}.

\begin{remark}
Observe that combining the Weyl law for $1$-cycles in $3$-manifolds from \cite{GL22} with Theorem \ref{thm3} which we just proved, we obtain Theorem \ref{thm2}.
\end{remark}

\section{Equidistribution of almost embedded closed geodesics in 2-manifolds}\label{closed geodesics}

In this section we show that the proof of Theorem \ref{thm3} combined with the work of Chodosh and Mantoulidis in \cite{Chodosh} (where they show that the $p$-widths on a surface are realized by collections of almost embedded closed geodesics) imply Theorem \ref{thm1}. The strategy to show this result will be to follow the proof of Theorem \ref{thm3} replacing ``embedded stationary geodesic network'' by ``almost embedded closed geodesic''. The main change needed in the proof is the following version of Proposition \ref{Perturbation}:

\begin{proposition}
Let $M$ be a closed $2$-manifold. Let $g:I^{N}\to\mathcal{M}^{q}$ be a smooth embedding, $N\in\mathbb{N}$. If $q\geq N+3$, there exists an arbitrarily small perturbation in the $C^{\infty}$ topology $g':I^{N}\to\mathcal{M}^{q}$ such that there is a full measure subset $\mathcal{A}\subseteq I^{N}$ with the following property: for any $p\in\mathbb{N}$ and any $t\in\mathcal{A}$, the function $s\mapsto\omega^{1}_{p}(g'(s))$ is differentiable at $t$ and there exist almost embedded closed geodesics $\gamma_{p}^{1},...,\gamma_{p}^{P}:S^{1}\to M$ such that the following two conditions hold
\begin{enumerate}
    \item $\omega_{p}^{1}(g'(t))=\sum_{i=1}^{P}L_{g'(t)}(\gamma_{p}^{i}(t))$.
    \item $\frac{\partial}{\partial v}(\omega^{1}_{p}\circ g')\big|_{s=t}=\frac{1}{2}\sum_{i=1}^{P}\int_{\gamma_{p}^{i}}\trace_{\gamma_{p}^{i},g'(t)}\frac{\partial g'}{\partial v}(t)\dL_{g'(t)}$.
\end{enumerate}
\end{proposition}

\begin{proof}
    We are going to adapt the proof of Proposition \ref{Perturbation} by introducing some necessary changes. A priori, the easiest way to do this seems to be substituting ``stationary geodesic network'' by ``finite union of almost embedded closed geodesics'' everywhere and use the Bumpy metrics theorem for almost embedded minimal submanifolds proved by Brian White in \cite{WhiteCinfty}. Nevertheless, there is not an easy condition (analog to conditions (1) to (7) in the proof of Proposition \ref{Perturbation}) that we can impose on a sequence of almost embedded closed geodesics to converge to another almost embedded closed geodesic without classifying them by their self-intersections and the angles formed there. Therefore, what we will do is to treat the almost embedded closed geodesics as a certain class of stationary geodesic networks, and then proceed as with Proposition \ref{Perturbation}.
    
    To each almost embedded closed geodesic $\gamma:S^{1}\to M$ we can associate a connected graph $\Gamma =S^{1}/\sim$ where $\sim$ is the equivalence relation $s\sim t$ if and only if $\gamma (s)=\gamma (t)$. This induces a map $f:\Gamma\to M$ defined as $f([t])=\gamma(t)$. Observe that the as the self-intersections of $\gamma$ are transverse, the vertices of $\Gamma$ are mapped precisely to those self-intersections and the map $f:\Gamma\to M$ is injective. Moreover, $\Gamma$ is a good multigraph and $f:\Gamma\to (M,g)$ is an embedded stationary geodesic network. We replace the set $\{\Gamma_{i}\}_{i\in\mathbb{N}}$ which in the proof of Proposition \ref{Perturbation} is the set of all good connected multigraphs by the countable set of pairs $\mathcal{P}=\{(\Gamma,r)\}$ where $\Gamma$ is a good multigraph which can be obtained as $\Gamma=S^{1}/\sim$ from an almost embedded closed geodesic $\gamma:S^{1}\to(M,g)$ with respect to some metric $g$ as before and $r$ is the set of pairs $((E_{1},i_{1}),(E_{2},i_{2}))$ such that $\pi_{E_{1}}(i_{1})=\pi_{E_{2}}(i_{2})$ and $(-1)^{i_{1}+1}\frac{\dot{f}_{E_{1}}(i_{1})}{|\dot{f}_{E_{1}}(i_{1})|_{g}}=(-1)^{i_{2}}\frac{\dot{f}_{E_{2}}(i_{2})}{|\dot{f}_{E_{2}}(i_{2})|_{g}}$ (in other words, $r$ contains the necessary information to reparametrize the geodesic net $f:\Gamma\to M$ to an immersed closed geodesic $\gamma:S^{1}\to (M,g)$). Observe that if $(\Gamma,r)\in\mathcal{P}$ and $f:\Gamma\to (M,g)$ is an embedded stationary geodesic network verifying $(-1)^{i_{1}+1}\frac{\dot{f}_{E_{1}}(i_{1})}{|\dot{f}_{E_{1}}(i_{1})|_{g}}=(-1)^{i_{2}}\frac{\dot{f}_{E_{2}}(i_{2})}{|\dot{f}_{E_{2}}(i_{2})|_{g}}$ for every $((E_{1},i_{1}),(E_{2},i_{2}))\in r$ then $f:\Gamma\to (M,g)$ can be reparametrized as an immersed closed geodesic $\gamma:S^{1}\to (M,g)$ whose self intersections occur precisely at the points $\{f(v):v\text{ vertex of }\Gamma\}$. 
    
    Taking the previous into account, instead of the $\tilde{B}_{\Gamma,M}$ in the proof of Proposition \ref{Perturbation} we will work with the following. Consider the set of pairs $(\Gamma,r)$ where $\Gamma$ is a graph, $\Gamma=\bigcup_{i=1}^{P}\Gamma_{i}$ as a union of connected components, $r=(r_{i})_{1\leq i\leq P}$ and $(\Gamma_{i},r_{i})\in\mathcal{P}$ for every $1\leq i\leq P$. Given such a pair $(\Gamma,r)$ and a natural number $M\in\mathbb{N}$ we define $\mathcal{B}_{\Gamma,r,M}$ to be the set of all $t\in (-1,1)^{N}$ such that there exists a stationary geodesic network $f:\Gamma\to (M,g'(t))$ verifying
    \begin{enumerate}
    \item For each $1\leq i\leq P$, $f_{i}=f|_{\Gamma_{i}}$ is an embedding and verifies the relations \linebreak $(-1)^{i_{1}+1}\frac{\dot{f}_{i,E_{1}}(i_{1})}{|\dot{f}_{i,E_{1}}(i_{1})|_{g'(t)}}=(-1)^{i_{2}}\frac{\dot{f}_{i,E_{2}}(i_{2})}{|\dot{f}_{i,E_{2}}(i_{2})|_{g'(t)}}$ for every $((E_{1},i_{1}),(E_{2},i_{2}))\in r_{i}$.
    \item $\Vert f_{i}\Vert_{3}\leq M$ for every $1\leq i\leq P$.
    \item $F_{1}(g'(t),f_{i})\geq\frac{1}{M}$ for every $1\leq i\leq P$.
    \item $F_{2}^{(E_{1},i_{1}),(E_{2},i_{2})}(g'(t),f_{i})\leq 1-\frac{1}{M}$ for every $1\leq i\leq P$,  and every pair $(E_{1},i_{1})\neq (E_{2},i_{2})\in\mathscr{E}_{i}\times\{0,1\}$ such that $\pi_{E_{1}}(i_{1})=\pi_{E_{2}}(i_{2})$.
    \item $d^{E}_{(g'(t),f_{i})}(s)\geq\frac{1}{M}$ for every $1\leq i\leq P$, $E\in\mathscr{E}_{i}$ and $s\in E$.
    \item $d^{E,E'}_{(g'(t),f_{i})}(s)\geq\frac{1}{M}$ for every $1\leq i\leq P$, $E\neq E'\in\mathscr{E}_{i}$ and $s\in E$.
    \item $\omega_{p}^{1}(g'(t))=\length_{g'(t)}(f)$.
\end{enumerate}

Therefore, same as in Proposition \ref{Perturbation} we have $I^{N}=\bigcup_{\Gamma,r,M}\mathcal{B}_{\Gamma,r,M}$ because of the fact showed in \cite{Chodosh} that the $p$-widths on surfaces are realized by unions of almost embedded closed geodesics; and each $\mathcal{B}_{\Gamma,r,M}$ is closed. The rest of the proof follows exactly as in Proposition \ref{Perturbation} if we replace the pairs $(\Gamma,M)$ by the triples $(\Gamma,r,M)$.
\end{proof}

One more remark is necessary to adapt the proof of Proposition \ref{PropP}. The sequences $(\gamma_{j,m})_{m}$ in (\ref{limit nets}) have length uniformly bounded by some $L>0$ and consist of finite unions of almost embedded closed geodesics. This implies that the number of closed geodesics whose union is $\gamma_{j,m}$ is also bounded (independently on $m$). Thus by applying Arzela-Ascoli to each of those components we can get a subsequence whose limit is not only a stationary geodesic net but also a union of closed curves with uniform convergence in $C^{0}$. The rest of the proof follows that of Theorem \ref{thm3} word for word.

\bibliography{Bibliography}
\bibliographystyle{amsplain}

\end{document}